\documentclass[12pt,leqno,a4paper]{amsart}
\usepackage{amsmath, verbatim, color}
\usepackage{mathrsfs}
\usepackage{dsfont}
\usepackage[a4paper,margin=1in]{geometry}
\usepackage{marginnote}

\newcommand{\normal}{\color{black}}

\theoremstyle{plain}
\newtheorem{theorem}{Theorem}[section]
\newtheorem{corollary}[theorem]{Corollary}

\theoremstyle{definition}
\newtheorem{remark}[theorem]{Remark}
\newtheorem{example}[theorem]{Example}

\newtheorem*{ack}{Acknowledgement}

\numberwithin{equation}{section}

\newcommand{\taus}{\tau}

\newcommand\E{\mathds E}

\newcommand\N{\mathds N}
\newcommand\R{\mathds R}
\newcommand\C{\mathds C}

\newcommand\I{\mathds 1}
\renewcommand\P{\mathds P}

\newcommand\Bscr{\mathscr{B}}

\newcommand\supp{\operatorname{supp}}
\newcommand\RE{\operatorname{Re}}

\renewcommand\d{\mathrm{d}}
\newcommand\e{\mathrm{e}}
\renewcommand\i{\mathop{\mathrm{i}}}

\begin{document}\allowdisplaybreaks
\title[Subordinate Shift Harnack Inequalities]{\bfseries On Shift Harnack Inequalities for Subordinate Semigroups and Moment Estimates for L\'evy Processes}

\author[C.-S.~Deng]{Chang-Song Deng}
\address[C.-S.~Deng]{School of Mathematics and Statistics\\ Wuhan University\\ Wuhan 430072, China}
\curraddr{TU Dresden\\ Fachrichtung Mathematik\\ Institut f\"{u}r Mathematische Stochastik\\ 01062 Dresden, Germany}
\email{dengcs@whu.edu.cn}
\thanks{The first-named author gratefully acknowledges support through the Alexander-von-Humboldt foundation, the National Natural Science Foundation of China (11401442) and the International Postdoctoral Exchange Fellowship Program (2013)}

\author[R.\,L.~Schilling]{Ren\'e L.\ Schilling}
\address[R.\,L.~Schilling]{TU Dresden\\ Fachrichtung Mathematik\\ Institut f\"{u}r Mathematische Stochastik\\ 01062 Dresden, Germany}
\email{rene.schilling@tu-dresden.de}


\date{\today}

\maketitle

\begin{abstract}
    We show that shift Harnack type inequalities (in the sense of F.-Y.~Wang \cite{Wan14}) are preserved
    under Bochner's subordination. The proofs are based on two types of moment estimates for subordinators. As a by-product we establish moment estimates for general L\'evy processes.
\end{abstract}

\section{Introduction and motivation}\label{sec1}

Subordination in the sense of Bochner is a method to generate new (`subordinate') stochastic processes from a given process by a random time change with an independent one-dimensional increasing L\'evy process (`subordinator'). A corresponding notion exists at the level of semigroups.
If the original process is a L\'evy process, so is the subordinate process. For instance, any symmetric $\alpha$-stable L\'evy process can be regarded as a subordinate Brownian motion, cf.~\cite{SSV}. This provides us with another approach to investigate jump processes via the corresponding results for diffusion processes. See \cite{GRW} for the dimension-free Harnack inequality for subordinate semigroups, \cite{SW} for subordinate functional inequalities and \cite{DS14} for the quasi-invariance property under subordination. In this paper, we will establish shift Harnack inequalities, which were introduced in \cite{Wan14}, for subordinate semigroups.

Let $(S_t)_{t\geq0}$ be a subordinator. Being a one-sided L\'evy process, it is uniquely determined by its Laplace transform which is of the form
$$
    \E\,\e^{-uS_t}
    =\e^{-t\phi(u)},\quad u>0,\;t\geq0.
$$
The characteristic exponent $\phi:(0,\infty) \to (0,\infty)$ is a Bernstein function having the following L\'evy--Khintchine representation
\begin{equation}\label{bern}
  \phi(u)
  =bu + \int_{(0,\infty)}\left(1-\e^{-ux}\right) \nu(\d x),
  \quad u>0.
\end{equation}
The drift parameter $b\geq 0$ and the L\'evy measure $\nu$---a measure on $(0,\infty)$ satisfying $\int_{(0,\infty)}(x\wedge1 )\,\nu(\d x)<\infty$---uniquely characterize the Bernstein function. The corresponding transition probabilities $\mu_t:=\P(S_t\in\cdot)$ form a vaguely continuous convolution semigroup of probability measures on $[0,\infty)$, i.e.\ one has $\mu_{t+s}=\mu_t*\mu_s$ for all $t,s\geq0$
and $\mu_t$ tends weakly to the Dirac measure concentrated at $0$ as $t\to 0$.

If $(X_t)_{t\geq 0}$ is a Markov process with transition semigroup $(P_t)_{t\geq 0}$, then the subordinate process is given by the random time-change $X_t^\phi := X_{S_t}$. The process $(X^\phi_t)_{t\geq 0}$ is again a Markov process, and it is not hard to see that the subordinate semigroup is given by the Bochner integral
\begin{equation}\label{sur2w}
    P_t^\phi f:=\int_{[0,\infty)}P_sf\,\mu_t(\d s),\quad
    t\geq 0,\; f\text{ bounded, measurable}.
\end{equation}

The formula \eqref{sur2w} makes sense for any Markov semigroup $(P_t)_{t\geq 0}$ on any Banach space $E$ and defines
again a Markov semigroup. We refer to \cite{SSV} for details, in particular for a functional calculus for the generator of $(P_t^\phi)_{t\geq0}$. If $(S_t)_{t\geq0}$ is an $\alpha$-stable subordinator ($0<\alpha<1$), the dimension-free Harnack type inequalities in the sense of \cite{Wan97} were established in \cite{GRW}, see \cite{Wbook} for more details on such Harnack inequalities. For example, if $(P_t)_{t\geq0}$ satisfies the log-Harnack inequality
$$
    P_t\log f(x)\leq\log P_tf(y)+\Phi(t,x,y),
    \quad x,y\in E,\, t>0,\, 1\leq f\in\Bscr_b(E)
$$
for some function $\Phi:(0,\infty)\times E\times E \to [0,\infty)$, then a similar inequality holds for the subordinate semigroup $(P_t^\phi)_{t\geq0}$; that is, the log-Harnack inequality is preserved under subordination.
For the stability
of the power-Harnack inequality, we need an additional condition on
$\alpha$: if the following power-Harnack inequality holds
$$
    \big(P_tf(x)\big)^p\leq\big(P_tf^p(y)\big)
    \exp\left[\Phi(t,p,x,y)\right],\quad x,y\in E, \,t>0,
    \, 0\leq f\in\Bscr_b(E),
$$
where $p>1$ and $\Phi(\cdot,p,x,y):(0,\infty)\to[0,\infty)$
is a measurable function such that for some $\kappa>0$
$$
    \Phi(t,p,x,y)=O(t^{-\kappa})\quad \text{as $t\to0$},
$$
then $(P_t^\phi)_{t\geq0}$ satisfies also a power-Harnack inequality provided that
$\alpha\in\left(\kappa/(1+\kappa),1\right)$,
see \cite[Theorem 1.1]{GRW}. We stress that the results of \cite{GRW}
hold for any subordinator whose Bernstein function
satisfies $\phi(u)\geq Cu^\alpha$ for large values of $u$ with
some constant $C>0$ and $\alpha\in(0,1)$
(as before, $\alpha\in\left(\kappa/(1+\kappa),1\right)$
is needed for the power-Harnack inequality),
see \cite[Proof of Corollary 2.2]{WW14}.

Recently, new types of Harnack inequalities, called shift Harnack inequalities, have been proposed in \cite{Wan14}: A Markov semigroup $(P_t)_{t\geq 0}$ satisfies the shift log-Harnack inequality if
\begin{equation}\label{shift-log-Harnack}
    P_t\log f(x)
    \leq \log P_t[f(\cdot+e)](x)+ \Psi(t,e),\quad  t>0,\, 1\leq f\in\Bscr_b(E),
\end{equation}
and the shift power-Harnack inequality with power $p>1$ if
\begin{equation}\label{shift-power-Harnack}
    \big(P_tf(x)\big)^p
    \leq \big(P_t[f^p(\cdot+e)](x)\big)
    \exp[\Phi(t,p,e)],\quad t>0,\, 0\leq f\in\Bscr_b(E);
\end{equation}
here, $x,e\in E$ and $\Psi(\cdot,e), \Phi(\cdot,p,e):(0,\infty) \to [0,\infty)$ are measurable functions. These new Harnack type inequalities can be applied to heat
kernel estimates and quasi-invariance properties
of the underlying transition probability under shifts,
see \cite{Wan14, Wbook} for details.

The power-Harnack and log-Harnack inequalities were established in \cite{WW14} for a class of stochastic differential equations driven by additive noise containing a subordinate Brownian motion as additive component, see also \cite{Den14} for an improvement and \cite{WZ15} for further developments. The shift power-Harnack inequality was also derived in \cite{Wan13} for this model by using the Malliavin's calculus and finite-jump approximations. In particular, when the noise is an $\alpha$-stable process with $\alpha\in(1,2)$, an explicit shift power-Harnack inequality was presented, cf.\ \cite[Corollary 1.4\,(3)]{Wan13}. However, shift log-Harnack inequalities for stochastic differential equations driven by jump processes have not yet been studied.

We want to consider the stability of shift Harnack inequalities under subordination. In many specific cases, see e.g.\ Example \ref{ggv6} in Section~\ref{sec2} below, $\Psi(s,e)$ and $\Phi(s,p,e)$ are of the form $C_1s^{-\kappa_1}+C_2s^{\kappa_2}+C_3$, with constants $C_i\geq0$, $i=1,2,3$, depending only on $e\in E$ and $p>1$, and exponents $\kappa_1, \kappa_2>0$. As it turns out, this means that we have to control $\E S_t^{\kappa}$, $\kappa\in\R\setminus\{0\}$, for the shift log-Harnack inequality and $\E\,\e^{\lambda S_t^\kappa}$, $\lambda>0$ and $\kappa\in\R\setminus\{0\}$, for the shift power-Harnack inequality. Throughout the paper, we use the convention that $\frac10:=\infty$.

Since moment estimates for stochastic processes are interesting on their own, we study such (exponential) moment estimates first for general L\'evy processes, and then for subordinators. For real-valued L\'evy processes without Brownian component estimates for the $p$th ($p>0$) moment were investigated in \cite{Mil71} and \cite{LP08} via the Blumenthal--Getoor index introduced in \cite{BG61}. While the focus of these papers were short-time asymptotics, we need estimates also for $t\gg 1$ which require a different set of indices for the underlying processes.

Let us briefly indicate how the paper is organized. In Section~\ref{sec2} we establish the shift Harnack type inequalities for the subordinate semigroup $P_t^\phi$ from the corresponding inequalities for $P_t$. Some practical conditions are presented to ensure the stability of the shift Harnack inequality under subordination; in Example \ref{ggv6} we illustrate our results using a class of stochastic differential equations. Section~\ref{sec3} is devoted to moment estimates of L\'evy processes: Subsection~\ref{sec31} contains, under various conditions, several concrete (non-)existence results and estimates for moments, while Subsection~\ref{sec32} provides the estimates for $\E S_t^\kappa$ and $\E\,\e^{\lambda S_t^\kappa}$ which were used in Section~\ref{sec2}. Finally, we show in Section~4 that the index $\beta_0$ of a L\'{e}vy symbol at the origin satisfies $\beta_0\in[0,2]$. As usual, we indicate by subscripts that a positive constant $C = C_{\alpha, \beta, \gamma, \dots}$ depends on the parameters $\alpha, \beta, \gamma, \dots$

\section{Shift Harnack inequalities for subordinate semigroups}\label{sec2}

In this section, we use the moment estimates
for subordinators from Section~3 to establish shift
Harnack inequalities for subordinate semigroups. Let $(P_t)_{t\geq0}$ be a Markov semigroup on a Banach space $E$ and $S=(S_t)_{t\geq0}$ be a subordinator whose characteristic (Laplace) exponent $\phi$ is the Bernstein function given by \eqref{bern}. Recall that the subordinate semigroup $(P^\phi_t)_{t\geq0}$ is defined by \eqref{sur2w}.

First, we introduce two indices for subordinators:
\begin{equation}\label{be0}
    \sigma_0
    :=\sup\left\{\alpha:\, \lim_{u\downarrow0}\frac{\phi(u)}{u^\alpha} =0\right\},
\end{equation}
\begin{equation}\label{be1}
    \rho_\infty
    :=\sup\left\{\alpha:\,\liminf_{u\to\infty} \frac{\phi(u)}{u^\alpha}>0\right\}
    =\inf\left\{\alpha:\,\liminf_{u\to\infty} \frac{\phi(u)}{u^\alpha}=0\right\}.
\end{equation}
Since
\begin{equation}\label{hg321cf}
    \lim_{u\downarrow0}\frac{\phi(u)}{u}
    =\lim_{u\downarrow0}\phi'(u)
    =b+\lim_{u\downarrow0}\int_{(0,\infty)}x\e^{-ux}\,\nu(\d x)
    =b+\int_{(0,\infty)} x\,\nu(\d x)\in(0,\infty],
\end{equation}
it is clear that $\sigma_0\in[0,1]$. Moreover, the following formula holds, see \cite{DS14}:
\begin{equation}\label{jn6gvc}
    \sigma_0=\sup\left\{\alpha:\,
    \limsup_{u\downarrow0}\frac{\phi(u)}{u^\alpha}
    <\infty\right\}
    =\sup\left\{\alpha\leq1:\,\int_{y\geq1}
    y^\alpha\,\nu(\d y)<\infty\right\}.
\end{equation}
For any $\epsilon>0$, noting that
$$
    0\leq\frac{\phi(u)}{u^{1+\epsilon}}\leq\frac{b}{u^\epsilon}
    +\frac{1}{u^\epsilon}\int_{(0,1)}x\,\nu(\d x)
    +\frac{1}{u^{1+\epsilon}}\,\nu(x\geq1),\quad u>0,
$$
yields
$$
    \lim_{u\to\infty}\frac{\phi(u)}{u^{1+\epsilon}}
    =0,
$$
one has $\rho_\infty\leq1+\epsilon$. Since $\epsilon>0$ is arbitrary,
we conclude that $\rho_\infty\in[0,1]$.

\begin{remark}\label{equico}
    We will frequently use the condition $\liminf_{u\to\infty}\phi(u)u^{-\rho}>0$ for some $\rho\in(c,\rho_\infty]$, where $c\geq0$. This is clearly equivalent to either $c<\rho<\rho_\infty$ or $\rho=\rho_\infty>c$ and $\liminf_{u\to\infty}\phi(u)u^{-\rho_\infty}>0$.
\end{remark}

Assume that $P_t$ satisfies the following shift log-Harnack inequality
\begin{equation}\label{2him}
    P_t\log f(x)
    \leq\log P_t[f(\cdot+e)](x)+\frac{C_1(e)}{t^{\kappa_1}}
    +C_2(e)\,t^{\kappa_2}+C_3(e),\quad t>0,\,1\leq f\in\Bscr_b(E),
\end{equation}
for some $x,e\in E$; here $\kappa_1>0$, $\kappa_2\in(0,1]$, and $C_i(e)\geq0$, $i=1,2,3$, are constants depending only on $e$.
We are going to show that the subordinate semigroup $P_t^\phi$ satisfies a similar shift log-Harnack inequality. The following assumptions on the subordinator will be important:
\begin{enumerate}
    \item[\textbf{(H1)}]
        $\kappa_1>0$ and $\liminf_{u\to\infty}\phi(u)u^{-\rho}>0$ for some $\rho\in(0,\rho_\infty]$.\footnote{This is equivalent to
    either $0<\rho<\rho_\infty$ or
    $\rho=\rho_\infty>0$ and $\liminf_{u\to\infty} \phi(u)u^{-\rho_\infty}>0$,
    see Remark \ref{equico}.}
    \item[\textbf{(H2)}]
        $\kappa_2\in(0,1]$ satisfies $\int_{y\geq1}y^{\kappa_2}\,\nu(\d y)<\infty$.
    \item[\textbf{(H3)}]
        $\kappa_2\in(0,1]$ satisfies $\inf_{\theta\in[\kappa_2,1]}\int_{y>0}y^\theta\,\nu(\d y)<\infty$.
    \item[\textbf{(H4)}]
        $\kappa_2\in(0,\sigma)$, where
        $\sigma\in(0,\sigma_0]$ satisfies $\limsup_{u\downarrow0}\phi(u)
        u^{-\sigma}<\infty$.\footnote{In
        analogy to Remark \ref{equico}, this is
        equivalent to either
        $0<\sigma<\sigma_0$ or
        $\sigma=\sigma_0>0$ and $\limsup_{u\downarrow0}\phi(u)u^{-\sigma_0}<\infty$.}
\end{enumerate}
Note that by \eqref{jn6gvc}, \textbf{\upshape(H4)} is strictly
stronger than \textbf{\upshape(H2)}.

\begin{theorem}\label{Har1}
    Suppose that $P_t$ satisfies \eqref{2him}
    for some $x,e\in E$. In each of the following cases the subordinate semigroup $P_t^\phi$ satisfies also a shift log-Harnack inequality
    \begin{equation}\label{subo-shift}
        P_t^\phi\log f(x) \leq \log P_t^\phi [f(\cdot+e)](x) + \Psi(t,e),\quad t>0,\,1\leq f\in\Bscr_b(E).
    \end{equation}

    \medskip\noindent\textup{\bfseries a)}
    Assume \textbf{\upshape(H1)} and
    \textbf{\upshape(H2)}.
    Then \eqref{subo-shift} holds with $\Psi(t,e)$ of the form
    $$
    C_1(e)\frac{C_{\kappa_1,\rho}}{(t\wedge1)^{\kappa_1/\rho}}
    +C_2(e)C_{\kappa_2}(t\vee1)+C_3(e).
    $$

    \medskip\noindent\textup{\bfseries b)}
    Assume \textbf{\upshape(H1)} and
    \textbf{\upshape(H3)}.
    Then \eqref{subo-shift} holds with $\Psi(t,e)$ of the form
    $$
    C_1(e)\frac{C_{\kappa_1,\rho}}{(t\wedge1)^{\kappa_1/\rho}}
    +C_2(e)C_{\kappa_2}(t\vee1)^{\kappa_2}+C_3(e).
    $$

    \medskip\noindent\textup{\bfseries c)}
    Assume \textbf{\upshape(H1)} and
    \textbf{\upshape(H4)}.
    Then \eqref{subo-shift} holds with $\Psi(t,e)$ of the form
    \begin{gather*}
        C_1(e)\frac{C_{\kappa_1,\rho}}{(t\wedge1)^{\kappa_1/\rho}}
        +C_2(e)C_{\kappa_2,\sigma}(t\vee1)^{\kappa_2/\sigma}
        +C_3(e).
    \end{gather*}
\end{theorem}

\begin{proof}
    Because of \eqref{2him} the shift log-Harnack
    inequality \eqref{shift-log-Harnack} for $P_t$ holds with $\Psi(s,e)=C_1(e)s^{-\kappa_1}+C_2(e)s^{\kappa_2}+C_3(e)$. Note
    that \textbf{\upshape(H1)} implies
    $\phi(u)\to\infty$ as $u\to\infty$, hence excluding
    the compound Poisson subordinator, so
    $$
        \mu_t(\{0\})=\P(S_t=0)=0,\quad t>0.
    $$
    By Jensen's inequality we find for all $t>0$
    \begin{align*}
      P_t^\phi\log f(x)
      &=\int_{(0,\infty)}P_s\log f(x)\,\mu_t(\d s)\\
      &\leq\int_{(0,\infty)}\big[\log P_s[f(\cdot+e)](x) +\Psi(s,e)\big]\,\mu_t(\d s)\\
      &\leq\log\int_{(0,\infty)}P_s [f(\cdot+e)](x)\,\mu_t(\d s)+\int_{(0,\infty)}\Psi(s,e)\,\mu_t(\d s)\\
      &=\log P_t^\phi  [f(\cdot+e)](x)
      +\int_{(0,\infty)}\Psi(s,e)\,\mu_t(\d s).
    \end{align*}
    Therefore,
    $$
        P_t^\phi\log f(x)
        \leq\log P_t^\phi [f(\cdot+e)](x) + C_1(e)\E S_t^{-\kappa_1} + C_2(e)\E S_t^{\kappa_2} + C_3(e)
    $$
    and the desired estimates follow from the corresponding moment bounds in Section~3.
\end{proof}

Now we turn to the shift power-Harnack
inequality for $P_t^\phi$.
We will assume that for some $p>1$
and $x,e\in E$, $P_t$ satisfies the
following shift power-Harnack inequality
    \begin{align}\label{shift}
        \big(P_tf(x)\big)^p
        \leq \big(P_t [f^p(\cdot+e)](x)\big) &\exp\left[\frac{H_1(p,e)}{t^{\kappa_1}}
        +H_2(p,e)t^{\kappa_2}+H_3(p,e)\right],
        \\ \nonumber &\qquad\qquad\qquad\qquad\qquad
        t>0,\,0\leq f\in\Bscr_b(E),
    \end{align}
    where $\kappa_1>0$, $\kappa_2\in(0,1]$ and
    $H_i(p,e)\geq0$, $i=1,2,3$, are constants depending on $p$ and $e$.

\begin{theorem}\label{Har2}
    Assume that $P_t$ satisfies \eqref{shift}
    for some $p>1$ and $x,e\in E$,
    and $\liminf_{u\to\infty}\phi(u)u^{-\rho}>0$ for
    some $\rho\in\left(\kappa_1/(1+\kappa_1),\rho_\infty\right]$.\footnote{This
    is equivalent to either $\kappa_1/(1+\kappa_1)<\rho<\rho_\infty$
    or $\rho=\rho_\infty>\kappa_1/(1+\kappa_1)$ and
    $\liminf_{u\to\infty} \phi(u)u^{-\rho_\infty}>0$,
    see Remark \ref{equico}.} If
    $$
        \int_{y\geq1}\exp\left[\frac{rH_2(p,e)}{p-1} y^{\kappa_2}\right]\nu(\d y)<\infty
    $$
    holds for some $r>1$, then the subordinate
    semigroup $P_t^\phi$ satisfies the shift power-Harnack inequality \eqref{shift-power-Harnack} with an
    exponent $\Phi(t,p,e)$ given by
    \begin{equation}\label{hg66gh}
        \frac{C_{\kappa_1,\rho,p,r,e}}
                {(t\wedge1)^{\frac{\kappa_1}{\rho-
                (1-\rho)\kappa_1}}}
        +C_{\kappa_2,p,r,e}(t\vee1)
        +H_3(p,e).
    \end{equation}
\end{theorem}

\begin{proof}
    As in the proof of Theorem \ref{Har1},
    we see $\mu_t(\{0\})=0$ for any $t>0$.
    By \eqref{shift} and the H\"{o}lder inequality one has
    \begin{align*}
        &\big(P_t^\phi f(x)\big)^p
        =\left(\int_{(0,\infty)}P_sf(x)\,\mu_t(\d s)\right)^p\\
        &\leq \left(\int_{(0,\infty)}\big(P_s[f^p(\cdot+e)](x)\big)^{\frac1p} \exp\left[\frac{H_1(p,e)}{ps^{\kappa_1}}+
        \frac{H_2(p,e)}{p}s^{\kappa_2}+\frac{H_3(p,e)}{p}\right] \mu_t(\d s)\right)^p\\
        &\leq\left(\int_{(0,\infty)}P_s[f^p(\cdot+e)](x)\,\mu_t(\d s)\right)\\
        &\qquad\mbox{}\times\left(\int_{(0,\infty)}
            \exp\left[\frac{H_1(p,e)}{(p-1)s^{\kappa_1}}+\frac{H_2(p,e)}{p-1}s^{\kappa_2}+\frac{H_3(p,e)}{p-1}\right] \mu_t(\d s)\right)^{p-1}\\
        &=\big(P_t^\phi[f^p(\cdot+e)](x)\big)\,\e^{H_3(p,e)}
        \left(\E\exp\left[\frac{H_1(p,e)}{p-1}S_t^{-\kappa_1}+\frac{H_2(p,e)}{p-1}S_t^{\kappa_2}\right]\right)^{p-1}\\
        &\leq\big(P_t^\phi[f^p(\cdot+e)](x)\big)\,\e^{H_3(p,e)}\\
        &\qquad\mbox{}\times\left(\E\exp\left[\frac{rH_1(p,e)}{(r-1)(p-1)}S_t^{-\kappa_1}\right]\right)^{\frac{(r-1)(p-1)}{r}}
        \left(\E\exp\left[\frac{rH_2(p,e)}{p-1}S_t^{\kappa_2}\right]\right)^{\frac{p-1}{r}}.
    \end{align*}
    By Theorem \ref{hyv3}\,b) in Section~\ref{sec3},
    $$
        \left(\E\exp\left[\frac{rH_1(p,e)}{(r-1)(p-1)}S_t^{-\kappa_1}\right]\right)^{\frac{(r-1)(p-1)}{r}}
        \leq\exp\left[
        \frac{C_{\kappa_1,\rho,p,r,e}}
                {(t\wedge1)^{\frac{\kappa_1}{\rho-
                (1-\rho)\kappa_1}}}
        \right].
    $$
    On the other hand, it follows
    from Theorem \ref{hygvfj}\,a) below that
    $$
        \left(\E\exp\left[\frac{rH_2(p,e)}{p-1}S_t^{\kappa_2}\right]\right)^{\frac{p-1}{r}}
        \leq
        \exp\big[C_{\kappa_2,p,r,e}(t\vee1)\big].
    $$
    Combining the above estimates finishes the proof.
\end{proof}

\begin{remark}
  By Theorem \ref{hyv3}\,b) below, in \eqref{hg66gh}
    $C_{\kappa_1,\rho,p,r,e}$ can be explicitly
    given as a function of $p,r,e$. In order to
    write $C_{\kappa_2,p,r,e}$
    in \eqref{hg66gh} explicitly as a function
    of $p,r,e$, we need some
    additional condition.
    If $\int_{(0,1)}|y|^{\kappa_2}\,\nu(\d y)<\infty$, this
    can be easily realized by using Theorem \ref{hygvfj}\,b) below. If
    $\nu(y\geq1)=0$, then we
    can use \cite[Theorem 25.17]{Sato}: for any $\lambda>0$, $\kappa_2\in(0,1]$ and $t>0$
    $$
        \E\,\e^{\lambda S_t^{\kappa_2}}\leq
        \E\,\e^{\lambda(1\vee S_t)}\leq
        \e^\lambda\,\E\,\e^{\lambda S_t}
        =\exp\left[\lambda+
        t\left(b\lambda+\int_{(0,1)}
        \left(\e^{\lambda y}-1\right)\,\nu(\d y)
        \right)
        \right].
    $$
\end{remark}

\begin{example}\label{ggv6}
    Consider the following stochastic differential equation on $\R^d$:
    \begin{equation}\label{sto}
        \d X_t
        = l_t(X_t)\,\d t+\Sigma_t\,\d W_t,
    \end{equation}
    where $l:[0,\infty)\times\R^d\to\R^d$
    and $\Sigma:[0,\infty)\to\R^d\times\R^d$ are measurable
    and locally bounded, and $(W_t)_{t\geq0}$ is a standard $d$-dimensional Brownian motion. We assume the following conditions on $l$ and $\Sigma$:
    \begin{enumerate}
        \item[\textbf{(A1)}]
            There exists a locally bounded measurable function $K:[0,\infty)\to[0,\infty)$ such that
            $$
                |l_t(x)-l_t(y)|\leq K_t|x-y|,
                \quad x,y\in\R^d,\;t\geq0.
            $$

        \item[\textbf{(A2)}]
            For each $t\geq0$, the matrix $\Sigma_t$ is invertible and there exists a measurable function
            $\gamma:[0,\infty)\to(0,\infty)$ such that $\gamma\in L^2_{\text{loc}}([0,\infty))$ and $\left\|\Sigma_t^{-1}\right\|\leq\gamma_t$ for all $t\geq0$.
    \end{enumerate}
    It is well known that \textbf{(A1)} ensures that \eqref{sto} has for each starting point $X_0=x\in\R^d$ a unique solution $(X_t^x)_{t\geq 0}$ with infinite life-time. By $P_t$ we denote the associated Markov semigroup, i.e.
    $$
        P_tf(x)=\E f(X_t^x),
        \quad t\geq0,\; f\in\Bscr_b(\R^d),\; x\in\R^d.
    $$

    Assume that for
    some $\kappa_1\geq1$ and $\kappa_2\in(0,1]$
    \begin{align}\label{con4}
        \limsup_{t\to 0}\frac{1}{t^{2-\kappa_1}}\int_0^t
        \gamma_r^2(rK_r+1)^2\,\d r
        &<\infty,\\
    \label{con5}
        \limsup_{t\to\infty}\frac{1}{t^{2+\kappa_2}}\int_0^t
        \gamma_r^2(rK_r+1)^2\,\d r
        &<\infty.
    \end{align}
    Typical examples for $K$ and $\gamma$ satisfying \eqref{con4} and \eqref{con5} are
    \begin{itemize}
        \item
            $\gamma_s=1$ and $K_s=s^\theta\wedge1$
            for $\theta\leq0$. Then it is easy to see
            that \eqref{con4} is fulfilled with $\kappa_1\geq1$
            and \eqref{con5} is satisfied with
            $\kappa_2\in[(2\theta+1)\vee0,1]\setminus\{0\}$.
        \item
            $\gamma_s=\I_{\{0\}}(s)
            +s^\theta\I_{(0,\infty)}(s)$ for
            $\theta\in(-1/2,0]$ and $K_s=1\vee(\log s)$.
            Then \eqref{con4} holds with
            $\kappa_1\geq1-2\theta$ and
            \eqref{con5} holds with $\kappa_2\in(1+2\theta,1]$.
    \end{itemize}

    We are going to show that there exists a constant $C>0$ such that for all $t>0$ and $x,e\in\R^d$ the following shift log- and power-$(p>1)$-Harnack inequalities hold:
    \begin{gather*}
        P_t\log f(x)\leq\log P_t[f(\cdot+e)](x)
        +C|e|^2\left(\frac{1}{t^{\kappa_1}}+t^{\kappa_2}+1\right),
        \quad 1\leq f\in\Bscr_b(\R^d),
    \\
        \big(P_t f(x)\big)^p\leq \big(P_t[f^p(\cdot+e)](x)\big)
        \exp\left[\frac{Cp |e|^2}{p-1}\left(\frac{1}{t^{\kappa_1}}+t^{\kappa_2}+1\right)\right],
        \quad 0\leq f\in\Bscr_b(\R^d).
    \end{gather*}
    In particular, Theorems \ref{Har1} and \ref{Har2} can be applied.

    Although the proof of Example \ref{ggv6} relies on known arguments, see e.g.~\cite{Wan14, Wbook}, we include the complete proof for the convenience of the reader.

\begin{proof}[Proof of Example \ref{ggv6}]
    Fix $t>0$ and $x,e\in\R^d$. We adopt the new coupling argument from \cite{Wan14} (see also \cite{Wbook}) to construct another process $(Y_s^x)_{s\geq0}$ also starting from $x$ such that $Y_t^x-X_t^x=e$ at the fixed time $t$. The process $(Y_s^x)_{s\geq0}$ is the solution of the following equation
    \begin{equation}\label{po8b}
        \d Y_s^x
        =l_s(X_s^x)\,\d s + \Sigma_s\,\d W_s + \frac{e}{t}\,\d s,
        \quad Y_0^x=x.
    \end{equation}
    Clearly,
    \begin{equation}\label{hg0v}
        Y_r^x - X_r^x
        = \int_0^r\frac{e}{t}\,\d s
        = \frac{r}{t}\,e,
        \quad 0\leq r\leq t,
    \end{equation}
    and, in particular, $Y_t^x-X_t^x=e$. Rewrite \eqref{po8b} as
    $$
        \d Y_s^x
        = l_s(Y_s^x)\,\d s + \Sigma_s\,\d \widetilde{W}_s,
        \quad Y_0^x=x,
    $$
    where
    $$
        \widetilde{W}_s
        :=W_s+\int_0^s\Sigma_r^{-1}\left(l_r(X_r^x)-l_r(Y_r^x)+\frac{e}{t}\right)\d r,
        \quad 0\leq s\leq t.
    $$
    Let
    $$
        M_s
        :=\int_0^s\Sigma_r^{-1}\left(l_r(X_r^x)
        -l_r(Y_r^x)+\frac{e}{t}\right)\d W_r,\quad 0\leq s\leq t.
    $$
    Since it follows from \textbf{(A2)}, \textbf{(A1)} and \eqref{hg0v} that
    \begin{align*}
        \left|\Sigma_r^{-1}\left(l_r(X_r^x)-l_r(Y_r^x)+\frac{e}{t}\right)\right|
        &\leq\gamma_r\left|l_r(X_r^x)-l_r(Y_r^x)+\frac{e}{t}\right|\\
        &\leq\gamma_r\left(K_r\left|X_r^x-Y_r^x\right|+\frac{|e|}{t}\right)\\
        &=\frac{|e|}{t}\gamma_r(rK_r+1),\quad 0\leq r\leq t,
    \end{align*}
    the compensator of the martingale $M$ satisfies
    \begin{equation}\label{h8mc}
        \langle M\rangle _t
        =\int_0^t\left|\Sigma_r^{-1}\left(l_r(X_r^x)
        -l_r(Y_r^x)+\frac{e}{t}\right)\right|^2\d r
        \leq\frac{|e|^2}{t^2}\int_0^t\gamma_r^2(rK_r+1)^2\,\d r.
    \end{equation}
    Set
    $$
        R_t:=\exp\left[-M_t-\frac12\langle M\rangle _t\right].
    $$
    Novikov's criterion shows that $\E R_t=1$. By the Girsanov theorem, $(\widetilde{W}_s)_{0\leq s\leq t}$ is a $d$-dimensional Brownian motion under the probability measure $R_t\P$.

    To derive the shift log-Harnack inequality for $P_t$, we first note that \eqref{h8mc} implies
    \begin{align*}
        \log R_t&=-M_t-\frac12\langle M\rangle _t\\
        &=-\int_0^t\Sigma_r^{-1}\left(l_r(X_r^x)-l_r(Y_r^x)+\frac{e}{t}\right)\d \widetilde{W}_r + \frac12\langle M\rangle _t\\
        &\leq-\int_0^t \Sigma_r^{-1}\left(l_r(X_r^x)-l_r(Y_r^x)+\frac{e}{t}\right)\d \widetilde{W}_r + \frac{|e|^2}{2t^2}\int_0^t \gamma_r^2(rK_r+1)^2\,\d r.
    \end{align*}
    Since $\E R_t=1$, we find with the Jensen inequality for any random variable $F\geq 1$
    $$
        \int R_t\log\frac{F}{R_t}\,\d\P \leq \log\int F\,\d\P,
        \quad\text{hence,}\quad
        \int R_t\log F\,\d\P \leq \log\int F\,\d\P + \int R_t\log R_t\,\d\P.
    $$
    Thus, we get for any $f\in\Bscr_b(\R^d)$ with $f\geq 1$
    \begin{equation}\label{hf3z}
    \begin{aligned}
        P_t\log f(x)&=\E_{R_t\P}\log f(Y_t^x)\\
        &=\E\big[R_t\log f(X_t^x+e)\big]\\
        &\leq\log\E f(X_t^x+e)+\E[R_t\log R_t]\\
        &=\log P_t[f(\cdot+e)](x)+\E_{R_t\P}\log R_t\\
        &\leq\log P_t[f(\cdot+e)](x)+
        \frac{|e|^2}{2t^2}\int_0^t
        \gamma_r^2(rK_r+1)^2\,\d r.
    \end{aligned}
    \end{equation}

    On the other hand, for any $p>1$ and $f\in\Bscr_b(\R^d)$ with $f\geq0$, we deduce with the H\"{o}lder inequality that
    \begin{align*}
        \big(P_tf(x)\big)^p&=\big(\E_{R_t\P}f(Y_t^x)\big)^p\\
        &=\big(\E[R_tf(X_t^x+e)]\big)^p\\
        &\leq\big(\E f^p(X_t^x+e)\big)
        \left(\E R_t^{\frac{p}{p-1}}\right)^{p-1}\\
        &=\big(P_t[f^p(\cdot+e)](x)\big)
        \left(\E R_t^{\frac{p}{p-1}}\right)^{p-1}.
    \end{align*}
    Because of \eqref{h8mc}, it follows that
    \begin{align*}
        \E R_t^{\frac{p}{p-1}}
        &=\E\exp\left[\frac{p}{2(p-1)^2}\langle M\rangle _t-\frac{p}{p-1}M_t-\frac{p^2}{2(p-1)^2}\langle M\rangle _t\right]\\
        &\leq\exp\left[\frac{p}{2(p-1)^2}\cdot\frac{|e|^2}{t^2}
        \int_0^t\gamma_r^2(rK_r+1)^2\,\d r\right]
        \E\exp\left[-\frac{p}{p-1}M_t-\frac{p^2}{2(p-1)^2}\langle M\rangle _t\right]\\
        &=\exp\left[\frac{p|e|^2}{2(p-1)^2t^2}\int_0^t\gamma_r^2(rK_r+1)^2\,\d r\right].
    \end{align*}
    In the last step we have used the fact that $\exp\left[-\frac{p}{p-1}M_t-\frac{p^2}{2(p-1)^2}\langle M\rangle _t\right]$ is a martingale; this is due to \eqref{h8mc} and Novikov's criterion. Therefore,
    \begin{equation}\label{j6cv}
        \big(P_tf(x)\big)^p
        \leq
        \big(P_t[f^p(\cdot+e)](x)\big) \exp\left[\frac{p|e|^2}{2(p-1)t^2}\int_0^t\gamma_r^2(rK_r+1)^2\,\d r\right]
    \end{equation}
    holds for all $p>1$ and non-negative $f\in\Bscr_b(\R^d)$.

    Finally, it remains to observe that \eqref{con4} and \eqref{con5} imply that there exists a constant $C>0$ such that
    $$
        \frac{1}{t^2}\int_0^t \gamma_r^2(rK_r+1)^2\,\d r
        \leq
        2C\left(\frac{1}{t^{\kappa_1}}+t^{\kappa_2}+1\right)
        \quad\text{for all $t>0$}.
    $$
    Substituting this into \eqref{hf3z} and \eqref{j6cv}, respectively, we obtain the desired shift log- and power-Harnack inequalities.
\end{proof}
\end{example}

\section{Moment estimates for L\'evy processes}\label{sec3}

\subsection{General L\'evy processes}\label{sec31}

A L\'evy process $X=(X_t)_{t\geq0}$ is a $d$-dimensional stochastic process with stationary and independent increments and almost surely c\`adl\`ag (right-continuous with finite left limits) paths $t\mapsto X_t$. As usual, we assume that $X_0=0$. Our standard references are \cite{Sato, Ja01}. Since L\'evy processes are (strong) Markov processes, they are completely characterized by the law of $X_t$, hence by the characteristic function of $X_t$. It is well known that
$$
    \E\e^{\i\xi\cdot X_t}=\e^{-t\psi(\xi)},
    \quad t>0,\;\xi\in\R^d,
$$
where the characteristic
exponent (L\'{e}vy symbol) $\psi:\R^d\to\C$ is given by the L\'evy--Khintchine formula
$$
    \psi(\xi)
    =-\i\ell\cdot\xi + \frac12\xi\cdot Q\xi +\int_{y\neq0} \left(1-\e^{\i\xi\cdot y} + \i\xi\cdot y \I_{(0,1)}(|y|)\right)\nu(\d y),
$$
where $\ell\in\R^d$ is the drift coefficient, $Q$ is a non-negative semidefinite $d\times d$ matrix, and $\nu$ is the L\'evy measure
on $\R^d\setminus\{0\}$ satisfying $\int_{y\neq 0}(1\wedge|y|^2)\,\nu(\d y)<\infty$. The L\'evy triplet $(\ell,Q,\nu)$ uniquely determines $\psi$, hence $X$. The infinitesimal generator
of $X$ is given by
$$
    \mathscr{L}f
    = \ell\cdot\nabla f +\frac12 \nabla\cdot Q\nabla f
    +\int_{y\neq0} \left(f(y+\cdot)-f- y\cdot\nabla f \I_{(0,1)}(|y|)\right)\nu(\d y),
    \quad
    f\in C_b^2(\R^d).
$$

First, we consider $\E|X_t|^\kappa$ for
$\kappa\in\R\setminus\{0\}$.
Recall, cf.\ \cite[Theorem 25.3]{Sato}, that the L\'evy process $X$ has a $\kappa$th ($\kappa>0$) moment, i.e.\ $\E|X_t|^\kappa<\infty$
for some (hence, all) $t>0$, if and only if
\begin{equation}\label{b6c2}
    \int_{|y|\geq1}|y|^\kappa\,\nu(\d y)<\infty.
\end{equation}
To present our results, we also need to introduce four indices
for L\'{e}vy processes. Define the Blumenthal--Getoor index
$\beta_\infty$ of a L\'{e}vy process by
$$
    \beta_\infty:=
    \inf\left\{\alpha\geq0\,:\,
    \lim_{|\xi|\to\infty}\frac{|\psi(\xi)|}{|\xi|^\alpha}
    =0\right\}.
$$
Clearly, $\beta_\infty\in[0,2]$. For a L\'{e}vy process
without diffusion part
and dominating drift, this coincides
with Blumenthal--and--Getoor's original definition based
on the L\'{e}vy measure $\nu$, cf.\ \cite{BG61, Sch98}, i.e.
$$
    \beta_\infty=
    \inf\left\{\alpha\geq0\,:\,
    \int_{0<|y|<1}|y|^\alpha\,\nu(\d y)<\infty\right\}.
$$
Moreover, it is easy to see that
\begin{equation}\label{ght56jn}
    \beta_\infty=\inf\left\{\alpha\geq0\,:\,
    \limsup_{|\xi|\to\infty}\frac{|\psi(\xi)|}{|\xi|^\alpha}
    =0\right\}
    =\inf\left\{\alpha\geq0\,:\,
    \limsup_{|\xi|\to\infty}\frac{|\psi(\xi)|}{|\xi|^\alpha}
    <\infty\right\}.
\end{equation}
In Section~\ref{sec2} we have introduced the index $\sigma_0$ for subordinators using the characteristic Laplace exponent (Bernstein function) $\phi$. A similar index exists for a general L\'evy process $X$---it is the counterpart at the origin of the classical Blumenthal--Getoor index---but its definition is based on the characteristic exponent $\psi$:
\begin{equation}\label{bet}
    \beta_0
    :=\sup\left\{\alpha\geq 0\,:\,\lim_{|\xi|\to 0}\frac{|\psi(\xi)|}{|\xi|^\alpha}=0\right\}.
\end{equation}
It holds that $\beta_0\in[0,2]$, see the
appendix (Section~4)
for an easy proof. As in \eqref{ght56jn}, one has
\begin{equation*}
    \beta_0
    =\sup\left\{\alpha\geq 0\,:\,\limsup_{|\xi|\to 0}\frac{|\psi(\xi)|}{|\xi|^\alpha}=0\right\}
    =\sup\left\{\alpha\geq 0\,:\,\limsup_{|\xi|\to 0}\frac{|\psi(\xi)|}{|\xi|^\alpha}<\infty\right\}.
\end{equation*}
Moreover, for a L\'{e}vy process without dominating drift,
it holds that, cf.\ \cite{Sch98},
\begin{equation*}
    \beta_0
    =\sup\left\{\alpha\leq 2\,:\, \int_{|y|\geq1}|y|^\alpha\,\nu(\d y)<\infty \right\}.
\end{equation*}
Let
$$
    \delta_\infty:=\sup\left\{\alpha\geq0\,:\,
    \liminf_{|\xi|\to\infty}\frac{\RE\psi(\xi)}
    {|\xi|^\alpha}>0
    \right\}
    =\inf\left\{\alpha\geq0\,:\,
    \liminf_{|\xi|\to\infty}\frac{\RE\psi(\xi)}
    {|\xi|^\alpha}=0
    \right\},
$$
$$
    \delta_0:=\inf\left\{\alpha\geq0\,:\,
    \liminf_{|\xi|\to0}\frac{\RE\psi(\xi)}
    {|\xi|^\alpha}>0
    \right\}
    =\sup\left\{\alpha\geq0\,:\,
    \liminf_{|\xi|\to0}\frac{\RE\psi(\xi)}
    {|\xi|^\alpha}=0
    \right\}.
$$
Clearly, $0\leq\delta_\infty\leq\beta_\infty\leq2$ and $\beta_0\leq\delta_0$. Note that $\delta_0\in[0,\infty]$.

\begin{theorem}\label{hb43}
    Let $X$ be a L\'evy process in $\R^d$ with L\'{e}vy
    triplet $(\ell,Q,\nu)$ and characteristic exponent
    $\psi$.

\medskip\noindent\textup{\bfseries a)}
    Assume that $Q=0$, $\kappa\in(0,1]$ and \eqref{b6c2} holds.
    Then for any $t>0$
     \begin{align*}
      \E|X_t|^\kappa
      &\leq |\ell|^\kappa t^\kappa +\left(\int_{|y|\geq1}|y|^\kappa\,\nu(\d y)\right)t\\
      &\qquad\mbox{}+2\left(\frac{d}{2}\kappa(3-\kappa)\int_{0<|y|<1}|y|^2\,\nu(\d y)\right)^{\kappa/2}
      \big[1+\nu(|y|\geq 1)\,t\big]^{1-\kappa/2} t^{\kappa/2}.
    \end{align*}
    In particular, $\E|X_t|^\kappa\leq C_\kappa\,t$ for any $t\geq1$.

\medskip\noindent\textup{\bfseries b)}
    Assume that $Q=0$, $\kappa\in(0,1]$ and
    $$
        \inf_{\theta\in[\kappa,1]}
        \int_{y\neq0}|y|^\theta\,\nu(\d y)<\infty.
    $$
   Then for any $t>0$
    \begin{equation}\label{ppb5c3}
        \E|X_t|^\kappa
        \leq
        \inf_{\theta\in[\kappa,1]} \left[|\hat{\ell}|^\theta t^\theta + \left(\int_{y\neq0}|y|^\theta\,\nu(\d y)\right)t \right]^{\kappa/\theta},
    \end{equation}
    where
    \begin{equation}\label{pi8h}
        \hat{\ell}:=\ell-\int_{0<|y|<1}y\,\nu(\d y).
    \end{equation}
    In particular, $\E|X_t|^\kappa\leq C_\kappa\,t^\kappa$ for
    any $t\geq1$.

\medskip\noindent\textup{\bfseries c)}
    Assume that $\kappa\in(0,\beta)$, where
    $\beta\in(0,\beta_0]$ satisfies
    $\limsup_{|\xi|\to0}
    |\psi(\xi)||\xi|^{-\beta}<\infty$.\footnote{This is equivalent to
    either $0<\beta<\beta_0$ or $\beta=\beta_0>0$ and
    $\limsup_{|\xi|\to0}
    |\psi(\xi)||\xi|^{-\beta_0}<\infty$.} Then for
    any $t\geq 1$
    $$
       \E|X_t|^\kappa\leq C_{\kappa,\beta,d}\,t^{\kappa/\beta}.
    $$
    If furthermore $\limsup_{|\xi|\to\infty}
        |\psi(\xi)||\xi|^{-\beta}<\infty$,\footnote{This is
        equivalent to
    either $\beta_\infty<\beta\leq\beta_0$ or
    $\beta=\beta_\infty\in(0,\beta_0]$ and $\limsup_{|\xi|\to\infty}
    |\psi(\xi)||\xi|^{-\beta_\infty}<\infty$.} then this
    estimate holds for all $t>0$.

\medskip\noindent\textup{\bfseries d)}
     Let $\kappa\in(0,d)$. If for
     some $\delta\in(0,\delta_\infty]$, $\liminf_{|\xi|\to\infty}
     \RE\psi(\xi)|\xi|^{-\delta}>0$,\footnote{This is
        equivalent to
    either $0<\delta<\delta_\infty$ or
    $\delta=\delta_\infty>0$ and
    $\liminf_{|\xi|\to\infty} \RE\psi(\xi)|\xi|^{-\delta_\infty}>0$.} then
    for any $t\in(0,1]$
     $$
         \E|X_t|^{-\kappa}
         \leq C_{\kappa,\delta,d}\,t^{-\kappa/\delta}.
     $$
     If furthermore $\liminf_{|\xi|\to0}
     \RE\psi(\xi)|\xi|^{-\delta}>0$,\footnote{This is
        equivalent to
    either $\delta_0<\delta\leq\delta_\infty$ or
    $\delta=\delta_0\in(0,\delta_\infty]$ and
    $\liminf_{|\xi|\to0} \RE\psi(\xi)|\xi|^{-\delta_0}>0$.}
    then this estimate holds for all $t>0$.

\medskip\noindent\textup{\bfseries e)}
    If $\psi$ is real-valued, then for any $\kappa\geq d$ and $t>0$
    $$
        \E|X_t|^{-\kappa}=\infty.
    $$

\end{theorem}

\begin{remark}\label{jn65cv}
    \textbf{\upshape a)}
        The technique (Taylor's theorem) used in the proof of
        Theorem \ref{hb43}\,a) below (see also
        the proof of Theorem \ref{hygvfj}\,a) below) can be used
        to get bounds on more general moments of
        L\'{e}vy processes,
        e.g.\ $\E|X_t|^\kappa$ with $\kappa>1$. However, the cost
        we have to pay is that the estimates may be too rough.

    \medskip\noindent\textbf{\upshape b)}
        It is well known that (cf.\ \cite[Theorem 25.3]{Sato})
the finiteness of moments of a L\'{e}vy process
depends only on the tail behaviour of
the underlying L\'{e}vy measure $\nu$, i.e.\ on the big jumps.
Theorem \ref{hb43}\,b)---it is essentially contained
in \cite[Theorem 2.1]{Mil71}, using a very different
method (Laplace transform)---requires also a condition on the small
jumps. This allows us to get a `clean' formula
for the $t$-dependence in \eqref{ppb5c3}.

\medskip\noindent\textbf{\upshape c)}
    Theorem \ref{hb43}\,b) is sharp for the Gamma subordinator,
    see Example \ref{exam} below.

\medskip\noindent\textbf{\upshape d)}
    Let $X$ be a symmetric $\alpha$-stable L\'evy process in $\R^d$ with $0<\alpha<2$. Then $\psi(\xi)=|\xi|^\alpha$ and
    we can choose in Theorem \ref{hb43}\,c)
    $\beta=\beta_0=\beta_\infty=\alpha$.
    For $t>0$ it is well known that $\E|X_t|^\kappa$ is finite if, and only if, $\kappa\in(0,\alpha)=(0,\beta)$, in which case $\E|X_t|^\kappa= t^{\kappa/\alpha} \E|X_1|^{\kappa}$. This means
    that Theorem \ref{hb43}\,c) is sharp
    for symmetric $\alpha$-stable L\'evy processes.

\medskip\noindent\textbf{\upshape e)}
    If the L\'{e}vy process has no dominating drift, then  $\limsup_{|\xi|\to0}|\psi(\xi)||\xi|^{-\beta}<\infty$ implies \eqref{b6c2} for $\kappa\in(0,\beta)$.
    The converse, however, is in general wrong.
    Thus the conditions in parts a) and c) of Theorem \ref{hb43} are incomparable.


\medskip\noindent\textbf{\upshape f)}
    For a one-dimensional symmetric $\alpha$-stable L\'evy process we have $\delta=\delta_\infty=\delta_0=\alpha$.
    By the scaling property, we have $\E|X_t|^{-\kappa}=C_{\kappa,\alpha}\,t^{-\kappa/\alpha}$ for $\kappa\in(0,1)$. Thus, Theorem \ref{hb43}\,d) is sharp for symmetric $\alpha$-stable L\'evy processes in $\R$.

\end{remark}

\begin{proof}[Proof of Theorem \ref{hb43}]
a)
   Rewrite $X_t$ as $X_t=\ell t+\widehat{X}_t$, $t\geq0$, where $\widehat X= (\widehat{X}_t)_{t\geq0}$ is the L\'evy process in $\R^d$
    generated by
    $$
        \widehat{\mathscr{L}}f
        =\int_{y\neq0} \left(f(y+\cdot)-f-
        y\cdot\nabla f\I_{(0,1)}(|y|)\right)\nu(\d y),
        \quad f\in C_b^2(\R^d).
    $$
    Noting that
    \begin{equation}\label{cr}
        (x+y)^a\leq x^a+y^a,
        \quad x,y\geq0,\,a\in[0,1],
    \end{equation}
    it suffices to show that for all $t>0$
    \begin{equation}\label{v5cs}
    \begin{aligned}
        \E|\widehat{X}_t|^\kappa
        &\leq \left(\int_{|y|\geq1}|y|^\kappa\,\nu(\d y)\right)t\\
        &\qquad\mbox{}+2\left(\frac{d}{2}\kappa(3-\kappa)\int_{0<|y|<1}|y|^2\,\nu(\d y)\right)^{\kappa/2}
        \big[1+\nu(|y|\geq1\,)t\big]^{1-\kappa/2}t^{\kappa/2}.
    \end{aligned}
    \end{equation}

    Fix $\epsilon>0$ and $t>0$. Let
    $$
        f(x):=\left(\epsilon+|x|^2\right)^{\kappa/2},
        \quad x\in\R^d,
    $$
    and
    \begin{equation}\label{j7v4}
        \taus_n:=\inf\left\{s\geq0\,:\,
        |\widehat{X}_s|>n
        \right\},\quad n\in\N.
    \end{equation}
    By the Dynkin formula we get for any $n\in\N$
    \begin{equation}\label{nhx9}
    \begin{aligned}
        \E&\left[\left(\epsilon+|\widehat{X}_{t\wedge\taus_n}|^2\right)^{\kappa/2}\right]-\epsilon^{\kappa/2}
        \:=\:\E\left[\int_{[0,t\wedge\taus_n)}\widehat{\mathscr{L}}f(\widehat{X}_s)\,\d s\right]\\
        &=\E\left[\int_{[0,t\wedge\taus_n)}\left(\int_{|y|\geq1}\left(f(\widehat{X}_s+y)-f(\widehat{X}_s)\right)\nu(\d y)\right)\d s\right]\\
        &\qquad\mbox{}+\E\left[\int_{[0,t\wedge\taus_n)}
        \left(\int_{0<|y|<1}\left(f(\widehat{X}_s+y)-f(\widehat{X}_s)- y\cdot \nabla f(\widehat{X}_s) \right)\nu(\d y)\right)\d s\right].
    \end{aligned}
    \end{equation}
    We estimate the two terms on the right side separately. For the first expression we have
    \begin{equation}\label{bx32}
        \int_{|y|\geq1}\left(f(\widehat{X}_s+y)-f(\widehat{X}_s)\right)\nu(\d y)
        \leq \epsilon^{\kappa/2}\nu(|y|\geq1)+ \int_{|y|\geq1}|y|^\kappa\,\nu(\d y).
    \end{equation}
    For the second term, we observe that for any $x\in\R^d$
    \begin{align*}
       \left|\frac{\partial^2f}{\partial x_j\partial x_i}(x)\right|
       &=\left|\kappa(\kappa-2)\left(\epsilon+|x|^2\right)^{\kappa/2-2}
       x_ix_j+\kappa\left(\epsilon+|x|^2\right)^{\kappa/2-1}\I_{\{i=j\}}\right|\\
       &\leq \kappa(2-\kappa)\left(\epsilon+|x|^2\right)^{\kappa/2-2}
       |x|^2+\kappa\left(\epsilon+|x|^2\right)^{\kappa/2-1}\\
       &\leq\kappa(2-\kappa)\epsilon^{\kappa/2-1}
       +\kappa\epsilon^{\kappa/2-1}\\
       &=\kappa(3-\kappa)\epsilon^{\kappa/2-1}.
    \end{align*}
    By Taylor's theorem,
    \begin{align*}
        f(\widehat{X}_s+y)-f(\widehat{X}_s)- y\cdot \nabla f(\widehat{X}_s)
        &=\frac12\sum_{i,j=1}^d\frac{\partial^2f}{\partial x_j\partial x_i}\big(\widehat{X}_s+\theta_{\widehat{X}_s,y}y\big) y_iy_j\\
        &\leq\frac12\kappa(3-\kappa)\epsilon^{\kappa/2-1}
        \sum_{i,j=1}^d|y_iy_j|\\
        &\leq\frac{d}{2}\kappa(3-\kappa)\epsilon^{\kappa/2-1}
        |y|^2,
    \end{align*}
    where $\theta_{\widehat{X}_s,y}\in[-1,1]$ depends on $\widehat{X}_s$ and $y$. Thus, we get
    $$
        \int_{0<|y|<1}\left(f(\widehat{X}_s+y)-f(\widehat{X}_s)- y\cdot \nabla f(\widehat{X}_s) \right)\nu(\d y)
        \leq\frac{d}{2}\kappa(3-\kappa)\epsilon^{\kappa/2-1}
        \int_{0<|y|<1}|y|^2\,\nu(\d y).
    $$
    Combining this with \eqref{bx32} and \eqref{nhx9}, we arrive at
    \begin{align*}
        \E\left[\big(\epsilon+|\widehat{X}_{t\wedge\taus_n}|^2\big)^{\kappa/2}\right]
        \leq\epsilon^{\kappa/2}
        &+\left(\epsilon^{\kappa/2}\nu(|y|\geq1) + \int_{|y|\geq1}|y|^\kappa\,\nu(\d y)\right) \E\left[t\wedge\taus_n\right]\\
        &+\left(\frac{d}{2}\kappa(3-\kappa)\epsilon^{\kappa/2-1}\int_{0<|y|<1}|y|^2\,\nu(\d y)\right) \E\left[t\wedge\taus_n\right].
    \end{align*}
    Since $\taus_n\uparrow\infty$ as $n\uparrow\infty$, we can let $n\uparrow\infty$ and use the monotone convergence theorem to obtain
    \begin{align*}
        \E|\widehat{X}_t|^\kappa
        &\leq \E\left[\big(\epsilon+|\widehat{X}_{t}|^2\big)^{\kappa/2}\right]\\
        &\leq\epsilon^{\kappa/2}
        +\left(\epsilon^{\kappa/2}\nu(|y|\geq1) +
        \int_{|y|\geq1}|y|^\kappa\,\nu(\d y) +
        \frac{d}{2} \kappa(3-\kappa)\epsilon^{\kappa/2-1} \int_{0<|y|<1}|y|^2\,\nu(\d y)\right)t\\
        &=\left(\int_{|y|\geq 1}|y|^\kappa\,\nu(\d y)\right)t\\
        &\qquad\mbox{}+\big[1+\nu(|y|\geq1)\,t\big]\epsilon^{\kappa/2} + \left[\frac{d}{2}\kappa(3-\kappa)\left(\int_{0<|y|<1}|y|^2\,\nu(\d y)\right)t\right]\epsilon^{\kappa/2-1}.
    \end{align*}
    Since $\epsilon>0$ is arbitrary, we can optimize
    in $\epsilon>0$, i.e.\  let
    $$
        \epsilon
        \downarrow
        \frac{\frac{d}{2}\kappa(3-\kappa)\left(\int_{0<|y|<1}|y|^2\,\nu(\d y)\right)t}{1+\nu(|y|\geq1)\,t},
    $$\normal
    to get \eqref{v5cs}.

\medskip\noindent b)
    Our assumption $\int_{0<|y|<1}|y|\,\nu(\d y)<\infty$ entails that $X$ has bounded variation. Therefore, $X_t=\hat{\ell}t+\widetilde{X}_t$, $t\geq0$, where $(\widetilde{X}_t)_{t\geq0}$ is a drift-free L\'evy process with generator
    $$
        \widetilde{\mathscr{L}}f=\int_{y\neq0}
        \big(f(y+\cdot)-f\big)\,\nu(\d y),
        \quad f\in C_b^2(\R^d).
    $$
    Let
    $$
        \sigma_n:=\inf\big\{t\,:\,|\widetilde{X}_t|>n\big\},\quad n\in\N.
    $$
    It follows from Dynkin's formula and \eqref{cr} that for any $\theta\in[\kappa,1]$ and $n\in\N$
    \begin{align*}
        \E|\widetilde{X}_{t\wedge\sigma_n}|^\theta
        =\E\left[\int_{[0,t\wedge\sigma_n)} \left(\int_{y\neq0}\left(|\widetilde{X}_{s}+y|^\theta-|\widetilde{X}_{s}|^\theta\right) \nu(\d y)\right)\d s\right]
        \leq \left(\int_{y\neq0}|y|^\theta\,\nu(\d y)\right)t.
    \end{align*}
    Since $\sigma_n\uparrow\infty$ as $n\uparrow\infty$, we can let $n\uparrow\infty$ and use the monotone convergence theorem to get
    $$
        \E|\widetilde{X}_{t}|^\theta
        \leq \left(\int_{y\neq0}|y|^\theta\,\nu(\d y)\right) t.
    $$
    Using \eqref{cr} again, we obtain that
    $$
        \E\left|X_{t}\right|^\theta
        \leq |\hat{\ell}|^\theta t^\theta + \E\big|\widetilde{X}_{t}\big|^\theta
        \leq |\hat{\ell}|^\theta t^\theta + \left(\int_{y\neq0}|y|^\theta\,\nu(\d y)\right) t.
    $$
    Together with Jensen's inequality, this yields for any $\theta\in[\kappa,1]$ and $t>0$
    $$
        \E|X_t|^\kappa\leq\left[\E|X_t|^\theta\right]^{\kappa/\theta}
        \leq\left[
        |\hat{\ell}|^\theta t^\theta
        +\left(\int_{y\neq0}|y|^\theta\,\nu(\d y)\right) t
        \right]^{\kappa/\theta},
    $$
    which implies the desired estimate.

\medskip\noindent c)
    Since $0<\kappa<\beta\leq\beta_0\leq2$, we have, see e.g.~\cite[III.18.23]{BF75},
    $$
        |x|^\kappa
        =c_{\kappa,d}\int_{\R^d\setminus\{0\}} \left(1-\cos(x\cdot \xi) \right)|\xi|^{-\kappa-d}\,\d\xi,\quad
        x\in\R^d,
    $$
    where
    $$
        c_{\kappa,d}:=\frac{\kappa 2^{\kappa-1}\Gamma\left(\frac{\kappa+d}
        {2}\right)}{\pi^{d/2}\Gamma\left(1-\frac{\kappa}{2}\right)}.
    $$
    By Tonelli's theorem, we get
    \begin{align*}
        \E|X_t|^\kappa
        &=c_{\kappa,d}\,\E\left[\int_{\R^d\setminus\{0\}} \left(1-\cos (X_t\cdot\xi) \right) |\xi|^{-\kappa-d}\,\d\xi\right]\\
        &=c_{\kappa,d}\int_{\R^d\setminus\{0\}} \left(1-\operatorname{Re}\E\,\e^{\i X_t\cdot\xi }\right) |\xi|^{-\kappa-d}\,\d\xi\\
        &=c_{\kappa,d}\int_{\R^d\setminus\{0\}} \left(1-\operatorname{Re}\e^{-t\psi(\xi)}\right) |\xi|^{-\kappa-d}\,\d\xi.
    \end{align*}
    Since $\operatorname{Re}\psi\geq 0$, we have
    $$
        \left|1-\operatorname{Re}\e^{-t\psi(\xi)}\right|
        \leq \left|1-\e^{-t\psi(\xi)}\right|
        \leq 2\wedge\big(t|\psi(\xi)|\big),
        \quad \xi\in\R^d\setminus\{0\}.
    $$
    If $\limsup_{|\xi|\to0}
    |\psi(\xi)||\xi|^{-\beta}<\infty$, then
    \begin{equation}\label{hhggfdv}
        |\psi(\xi)|\leq C_\beta|\xi|^\beta,\quad 0<|\xi|\leq 1.
    \end{equation}
    Thus, we find for all $t\geq 1$
    \begin{equation}\label{hg5dcc}
    \begin{aligned}
        \E|X_t|^\kappa
        &\leq c_{\kappa,d}\int_{\R^d\setminus\{0\}} \left[2\wedge\big(t\,|\psi(\xi)|\big)\right] |\xi|^{-\kappa-d}\,\d\xi\\
        &\leq c_{\kappa,d}\int_{0<|\xi|\leq t^{-1/\beta}} t\,|\psi(\xi)||\xi|^{-\kappa-d}\,\d\xi + c_{\kappa,d}\int_{|\xi|> t^{-1/\beta}}
        2\,|\xi|^{-\kappa-d}\,\d\xi\\
        &\leq c_{\kappa,d}C_\beta \,t \int_{0<|\xi|\leq t^{-1/\beta}} |\xi|^{\beta-\kappa-d}\,\d\xi
        + 2c_{\kappa,d} \int_{|\xi|> t^{-1/\beta}} |\xi|^{-\kappa-d}\,\d\xi\\
        &=c_{\kappa,d}C_\beta \,t^{\kappa/\beta} \int_{0<|\xi|\leq 1} |\xi|^{\beta-\kappa-d}\,\d\xi
        + 2c_{\kappa,d}\,t^{\kappa/\beta}\int_{|\xi|> 1} |\xi|^{-\kappa-d}\,\d\xi,
    \end{aligned}
    \end{equation}
    which implies the first estimate.
    If furthermore $\limsup_{|\xi|\to\infty}
    |\psi(\xi)||\xi|^{-\beta}<\infty$, then
    \eqref{hhggfdv} holds for all $\xi\in\R^d$, and
    so \eqref{hg5dcc} holds true for all $t>0$.
    This gives the second assertion.

\medskip\noindent d)
     Recall that
     $$
        \e^{-u|x|}
        =c_d \int_{\R^d}\frac{u}
        {\left(u^2+|\xi|^2\right)^{\frac{d+1}{2}}}
        \,\e^{-\i x\cdot\xi }
        \,\d\xi,
        \quad x\in\R^d,\,u>0,
    $$
    where
    $$
        c_d:=\pi^{-(d+1)/2}\,
        \Gamma\left(\tfrac{d+1}{2}\right).
    $$
    Using Fourier transforms, we have for all $u>0$ and $x\in\R^d$
    $$
         \int_{\R^d}|\xi|^{-d+\kappa}\,\e^{-u|\xi|}
         \,\e^{-\i x\cdot\xi}\,\d\xi
         =c_{\kappa,d}' \,c_d\int_{\R^d}|x-y|^{-\kappa}
         \frac{u}{\left(
         u^2+|y|^2
         \right)^{\frac{d+1}{2}}}\,\d y,
     $$
     where
     $$
         c_{\kappa,d}':=2^\kappa\pi^{d/2}
         \Gamma\left(\tfrac\kappa2\right) \Big/ \Gamma\left(\tfrac{d-\kappa}{2}\right).
     $$
     This implies that for all $n\in\N$ and $t>0$
     \begin{align*}
         \int_{\R^d}|\xi|^{-d+\kappa}\,\e^{-n^{-1}|\xi|}
         \RE\e^{\i X_t\cdot\xi}\,\d\xi
         &=c_{\kappa,d}' \,c_d
         \int_{\R^d}\left|X_t+y\right|^{-\kappa}
         \frac{n^{-1}}{\left(
         n^{-2}+|y|^2
         \right)^{\frac{d+1}{2}}}\,\d y\\
         &=c_{\kappa,d}' \,c_d
         \int_{\R^d}\left|X_t+n^{-1}y\right|^{-\kappa}
         \frac{\d y}{\left(
         1+|y|^2
         \right)^{\frac{d+1}{2}}}.
     \end{align*}
     Taking expectation and using Fubini's theorem, we get
     $$
         \int_{\R^d}|\xi|^{-d+\kappa}\,\e^{-n^{-1}|\xi|}
         \RE\e^{-t\psi(\xi)}\,\d\xi
         =c_{\kappa,d}' \,c_d\,
         \E\left[\int_{\R^d}\left|X_t+n^{-1}y\right|^{-\kappa}
         \frac{\d y}{\left(
         1+|y|^2
         \right)^{\frac{d+1}{2}}}\right].
     $$
     Combining this with Fatou's lemma and Tonelli's
     theorem, we obtain
     \begin{align*}
          \E|X_t|^{-\kappa}
          &=\E|X_t|^{-\kappa}\cdot c_d\int_{\R^d}\frac{\d y}{\left( 1+|y|^2 \right)^{\frac{d+1}{2}}}\\
          &=c_d\int_{\R^d} \E\left[ \liminf_{n\to\infty} \left|X_t+n^{-1}y\right|^{-\kappa} \right] \frac{\d y}{\left( 1+|y|^2 \right)^{\frac{d+1}{2}}}\\
          &\leq c_d\liminf_{n\to\infty}\E\left[ \int_{\R^d}\left|X_t+n^{-1}y\right|^{-\kappa} \frac{\d y}{\left( 1+|y|^2 \right)^{\frac{d+1}{2}}}\right]\\
          &=\frac{1}{c_{\kappa,d}'} \liminf_{n\to\infty} \int_{\R^d}|\xi|^{-d+\kappa}\,\e^{-n^{-1}|\xi|} \RE\e^{-t\psi(\xi)}\,\d\xi\\
          &\leq\frac{1}{c_{\kappa,d}'} \liminf_{n\to\infty} \int_{\R^d}|\xi|^{-d+\kappa}\,\e^{-n^{-1}|\xi|} \,\e^{-t\RE\psi(\xi)}\,\d\xi\\
          &=\frac{1}{c_{\kappa,d}'} \int_{\R^d}|\xi|^{-d+\kappa} \,\e^{-t\RE\psi(\xi)}\,\d\xi,
     \end{align*}
     where we have used the monotone convergence theorem. If $\liminf_{|\xi|\to\infty} \RE\psi(\xi)|\xi|^{-\delta}>0$, then there exist constants $C_1=C_1(\delta)>0$ and $C_2=C_2(\delta)\geq0$ such that
     \begin{equation}\label{h5fcbvb}
         \RE\psi(\xi)\geq C_1|\xi|^\delta,
         \quad |\xi|\geq C_2.
     \end{equation}
     Thus, we find for all $t>0$
     \begin{equation}\label{d343sdf}
    \begin{aligned}
         \E|X_t|^{-\kappa}&\leq\frac{1}{c_{\kappa,d}'}
         \left(
         \int_{|\xi|\leq C_2}|\xi|^{-d+\kappa}\,\d\xi
         +\int_{\R^d}|\xi|^{-d+\kappa}
         \,\e^{-tC_1|\xi|^\delta}\,\d\xi
         \right)\\
         &=\frac{1}{c_{\kappa,d}'}
         \left(
         \int_{|\xi|\leq C_2}|\xi|^{-d+\kappa}\,\d\xi
         +t^{-\kappa/\delta}\int_{\R^d}|\xi|^{-d+\kappa}
         \,\e^{-C_1|\xi|^\delta}\,\d\xi
         \right).
    \end{aligned}
    \end{equation}
    This gives the first assertion. If furthermore $\liminf_{|\xi|\to0}
     \RE\psi(\xi)|\xi|^{-\delta}>0$, then \eqref{h5fcbvb}
     holds with $C_2=0$, so that
     the second assertion follows
     by using \eqref{d343sdf} with $C_2=0$.

\medskip\noindent e)
   Using
    \begin{equation}\label{h6fc}
        \frac{1}{y^{r}}
        =\frac{1}{\Gamma(r)}\int_0^\infty\e^{-uy}u^{r-1}\,\d u,
        \quad r>0,\,y\geq 0
    \end{equation}
    and the Fourier transform formula from the beginning of part d), we get by Tonelli's theorem that
    \begin{align*}
        \E|X_t|^{-\kappa}
        &=\frac{1}{\Gamma(\kappa)}\, \E\left[\int_0^\infty\e^{-u|X_t|}u^{\kappa-1}\,\d u\right]\\
        &=\frac{c_d}{\Gamma(\kappa)}\int_0^\infty\E\left[\int_{\R^d}\e^{\i X_t\cdot\xi } \frac{u}{\left(u^2+|\xi|^2\right)^{\frac{d+1}{2}}}\,\d\xi
        \right] u^{\kappa-1}\,\d u.
    \end{align*}
    Since $\psi(\xi)\in\R$ for all $\xi\in\R^d$,
    $0<\e^{-t\psi(\xi)}\leq1$, and we can use Fubini's theorem
    for the inner integrals and then Tonelli's theorem for the
    two outer integrals to get
    \begin{align*}
        \E|X_t|^{-\kappa}
        &= \frac{c_d}{\Gamma(\kappa)} \int_0^\infty
        \left(\int_{\R^d} \e^{-t\psi(\xi)}\frac{u}{\left(u^2+|\xi|^2\right)^{\frac{d+1}{2}}}\,\d\xi\right)u^{\kappa-1}\,\d u\\
        &= \frac{c_d}{\Gamma(\kappa)} \int_{\R^d}
        \left(\int_0^\infty\frac{u^{\kappa}}{\left(u^2+|\xi|^2\right)^{\frac{d+1}{2}}}\,\d u\right)\e^{-t\psi(\xi)}\,\d\xi\\
        &\geq \frac{c_d}{\Gamma(\kappa)}\int_{\R^d}
        \left(\int_{|\xi|}^\infty\frac{u^{\kappa}}{\left(u^2+u^2\right)^{\frac{d+1}{2}}}\,\d u\right)\e^{-t\psi(\xi)}\,\d\xi\\
        &= \frac{c_d}{2^{\frac{d+1}{2}}\Gamma(\kappa)} \int_{\R^d}
        \left(\int_{|\xi|}^\infty u^{\kappa-d-1}\,\d u\right)\e^{-t\psi(\xi)}\,\d\xi
        =\infty,
    \end{align*}
    where the last equality follows from $\kappa\geq d$.
\end{proof}

Let us now consider $\E\,\e^{\lambda|X_t|^\kappa}$ for $\lambda>0$ and $\kappa\in\R\setminus\{0\}$. For $\kappa\in(0,1]$ the function $\R^d\ni x\mapsto\e^{\lambda|x|^\kappa} \in\R$ is submultiplicative; thus, $\E\,\e^{\lambda|X_t|^\kappa}<\infty$ for some (hence, all) $t>0$ if, and only if,
\begin{equation}\label{bt5cz}
    \int_{|y|\geq1}\e^{\lambda|y|^\kappa}
    \,\nu(\d y)<\infty,
\end{equation}
see~\cite[Theorem 25.3]{Sato}.

\begin{theorem}\label{hygvfj}
Let $X$ be a L\'evy process in $\R^d$ with L\'{e}vy
    triplet $(\ell,Q,\nu)$ and characteristic exponent
    $\psi$.

\medskip\noindent\textup{\bfseries a)}
    If $Q=0$, $\kappa\in(0,1]$
    and \eqref{bt5cz} holds, then for
    any $\lambda>0$ there is some (non-explicit)
    constant $C_{\kappa,\lambda}>0$ such that
    $$
        \E\,\e^{\lambda|X_t|^\kappa}
        \leq
        \begin{cases}
            \e^{C_{\kappa,\lambda}\,t^{\kappa/2}}, &\text{if\ \ } t<1,\\
            \e^{C_{\kappa,\lambda}\,t}, &\text{if\ \ } t\geq1.
        \end{cases}
    $$

\medskip\noindent\textup{\bfseries b)}
     If $Q=0$, $\kappa\in(0,1]$, \eqref{bt5cz} holds and
     \begin{equation}\label{j7ce}
        \int_{0<|y|<1}|y|^\kappa\,\nu(\d y)<\infty,
    \end{equation}
    then for any $\lambda,t>0$
    $$
      \E\,\e^{\lambda|X_t|^\kappa}
      \leq\exp\left[\lambda|\hat{\ell}|^{\kappa}t^{\kappa} + M_{\kappa,\lambda}\,t\right],
    $$
    where $\hat{\ell}$ is given by \eqref{pi8h} and
    $$
        M_{\kappa,\lambda}
        :=\int_{y\neq0}\big(
        \e^{\lambda|y|^\kappa}
        -1\big)\,\nu(\d y).
    $$

\medskip\noindent\textup{\bfseries c)}
    If $\nu\neq 0$ and $\kappa>1$, then for any $\lambda,t>0$
    $$
        \E\,\e^{\lambda|X_t|^\kappa}=\infty.
    $$

\medskip\noindent\textup{\bfseries d)}
    If $\psi$ is real-valued,
    then for any $\lambda,t,\kappa>0$
    $$
        \E\,\e^{\lambda|X_t|^{-\kappa}}=\infty.
    $$
\end{theorem}

\begin{remark}
    \textbf{\upshape a)}
        Since $\nu$ is a L\'evy measure, it is easy to see that \eqref{bt5cz} and \eqref{j7ce} imply $M_{\kappa,\lambda}<\infty$.

    \medskip\noindent\textbf{\upshape b)}
        It is well known, see e.g.\ \cite[Theorem 26.1\,(ii)]{Sato}, that
    $$
        \E\,\e^{\lambda|X_t|\log|X_t|}=\infty,\quad \lambda>0
    $$
    for any L\'evy process with L\'evy measure having unbounded support $\supp\nu$. Since for any $\kappa>1$ there exists a constant $C_\kappa > 0$ such that
    $$
        \e^{\lambda|x|\log|x|}\leq C_\kappa\e^{\lambda|x|^\kappa},
        \quad \lambda>0,\,x\in\R^d,
    $$
    this implies Theorem \ref{hygvfj}\,c) if $\supp\nu$ is
    unbounded; Theorem \ref{hygvfj}\,c), however, is valid for all non-zero $\nu$.
\end{remark}

\begin{proof}[Proof of Theorem \ref{hygvfj}]
a)
   Let $\widehat{X}$ be a L\'evy process in $\R^d$ with triplet $(0,0,\nu)$. It is enough to show that
    \begin{equation}\label{v5cx}
        \E\,\e^{\lambda|\widehat{X}_t|^\kappa}
        \leq \E\exp\left[\lambda\left(\epsilon+|\widehat{X}_t|^{2}\right)^{\kappa/2}\right]
        \leq \exp\left[\lambda\epsilon^{\kappa/2}\left(
        1+\varepsilon^{-1}C_1t\right)+C_2
        t\right]
    \end{equation}
    for all $\epsilon\in(0,1]$ and $t>0$, where
    $$
        C_1:=\frac{d}{2}\kappa(\lambda\kappa
        +3-\kappa)\e^{2\lambda}\int_{0<|y|<1}|y|^2\,\nu(\d y)
        ,\quad
        C_2:=\e^\lambda\int_{|y|\geq1}\e^{\lambda|y|^\kappa}\,
        \nu(\d y)-\nu\left(|y|\geq1\right).
    $$
    Fix $\epsilon\in(0,1]$ and $t>0$. Let
    $$
        g(x):=\exp\left[\lambda\left(\epsilon+|x|^2\right)^{\kappa/2}\right],
        \quad x\in\R^d,
    $$
    and define $\taus_n$ by \eqref{j7v4}. As in the proof of Theorem \ref{hb43}\,a) we can use a Taylor expansion to get for $s<t\wedge\taus_n$
    $$
        \int_{|y|\geq 1} \left(g(\widehat{X}_s+y)-g(\widehat{X}_s)\right)\,\nu(\d y)
        \leq C_2 g(\widehat{X}_s),
    $$
    $$
        \int_{0<|y|<1}\left(g(\widehat{X}_s+y) - g(\widehat{X}_s)- y\cdot \nabla g(\widehat{X}_s) \right)\,\nu(\d y)
        \leq C_1\lambda\epsilon^{\kappa/2-1}
        g(\widehat{X}_s).
    $$
    Now we use Dynkin's formula and Tonelli's theorem to obtain
    \begin{align*}
        \E\left[g(\widehat{X}_{t})\I_{\{t<\taus_n\}}\right]
        -\e^{\lambda\epsilon^{\kappa/2}}
        &\leq\E g(\widehat{X}_{t\wedge\taus_n}) - \e^{\lambda\epsilon^{\kappa/2}}\\
        &\leq \left(C_1\lambda\epsilon^{\kappa/2-1}+C_2\right)
        \E\left[\int_{[0,t\wedge\taus_n)} g(\widehat{X}_{s})\,\d s\right]\\
        &=\left(C_1\lambda\epsilon^{\kappa/2-1}+C_2\right)
        \int_0^t\E\left[g(\widehat{X}_{s})\I_{\{s<\taus_n\}}\right]\d s.
    \end{align*}
    From Gronwall's inequality we see
    $$
        \E\left[g(\widehat{X}_{t})\I_{\{t<\taus_n\}}\right]\leq\exp\left[\lambda\epsilon^{\kappa/2} + \left(C_1\lambda\epsilon^{\kappa/2-1}+C_2\right)t\right]
    $$
    for all $n\in\N$. Finally, \eqref{v5cx} follows as $n\uparrow\infty$.

\medskip\noindent b)
     As in the proof of Theorem \ref{hb43}\,b), we use Dynkin's formula, \eqref{cr} and Tonelli's theorem to obtain that for all $n\in\N$
    \begin{align*}
        \E\left[\e^{\lambda|\widetilde{X}_{t}|^\kappa}
        \I_{\{t<\sigma_n\}}\right]-1
        &\leq \E\,\e^{\lambda|\widetilde{X}_{t\wedge\sigma_n}|^\kappa}
        -1\\
        &=\E\left[\int_{[0,t\wedge\sigma_n)}
        \left(\int_{y\neq 0} \left(
        \e^{\lambda|\widetilde{X}_{s}+y|^\kappa}-
        \e^{\lambda|\widetilde{X}_{s}|^\kappa}
        \right)\,\nu(\d y)\right)\d s\right]\\
        &\leq\E\left[\int_{[0,t\wedge\sigma_n)}\left(\int_{y\neq 0}
        \e^{\lambda|\widetilde{X}_{s}|^\kappa}
        \left(
        \e^{\lambda|y|^\kappa}
        -1\right)\nu(\d y)\right)\d s\right]\\
        &=M_{\kappa,\lambda}\,\E\left[\int_{[0,t\wedge\sigma_n)}
        \e^{\lambda|\widetilde{X}_{s}|^\kappa}
        \d s\right]\\
        &=M_{\kappa,\lambda}
        \int_0^t \E\left[
        \e^{\lambda|\widetilde{X}_{s}|^\kappa}
        \I_{\{s<\sigma_n\}}\right]\d s.
    \end{align*}
    This, together with Gronwall's inequality, yields that
    $$
        \E\left[
        \e^{\lambda|\widetilde{X}_{t}|^\kappa}
        \I_{\{t<\sigma_n\}}\right]
        \leq \e^{M_{\kappa,\lambda}\,t},\quad n\in\N.
    $$
    It remains to let $n\uparrow\infty$ and
    use \eqref{cr} to get the desired result.

\medskip\noindent c)
     Since $\nu\neq 0$ we may, without loss of generality,
    assume that there exist some Borel set $B_1\subset\R$ with either
    $\inf B_1>0$ or $\sup B_1<0$ and Borel
    sets $B_2,\dots,B_d\subset\R\setminus\{0\}$ such that
    $$
        \eta:=\nu(\Lambda)\in(0,\infty),
    $$
    where $\Lambda := B_1\times B_2\times\dots\times B_d$. The jump times of jumps with size in the set $\Lambda$ define a Poisson process,
    say $(N_t)_{t\geq 0}$, with intensity $\eta$. Note that $X$ can be decomposed into two independent L\'evy processes
    $$
        X_t=X_t^{(1)}+X_t^{(2)},\quad t\geq0,
    $$
    where $X^{(1)}$ is a compound Poisson process with L\'evy measure $\nu|_\Lambda$, and $X^{(2)}$ is a L\'evy process with L\'evy measure $\nu-\nu|_\Lambda$; moreover, $X^{(1)}$ and $X^{(2)}$ are
    independent processes. Set
    $$
        r:=\left|\inf B_1\right|\wedge\left|\sup B_1\right|
        \in(0,\infty).
    $$
    By the triangle inequality we find for any $y\in\R^d$,
    $$
        |X_t^{(1)}+y|
        \geq |X_t^{(1)}|-|y|
        \geq r N_t - |y|.
    $$
    Using Stirling's formula
    $$
        n!
        \leq\sqrt{2\pi}n^{n+\frac12}\e^{-n+\frac{1}{12n}}
        \leq\sqrt{2\pi}n^{n+1}\e^{-n+1},
        \quad n\in\N,
    $$
    we obtain that for any $\lambda,t>0$
    \begin{align*}
        \E\,\e^{\lambda|X_t^{(1)}+y|^\kappa}
        &\geq \E\left[\e^{\lambda(rN_t-|y|)^\kappa}\I_{\{rN_t>|y|\}}\right]\\
        &= \sum_{n\,:\,rn>|y|}\e^{\lambda(rn-|y|)^\kappa}\,\frac{(\eta t)^n\e^{-\eta t}}{n!}\\
        &\geq\frac{\e^{-\eta t}}{\sqrt{2\pi}\,\e}
        \sum_{n\,:\,rn>|y|}\frac{(\eta\e t)^n\e^{\lambda(rn-|y|)^\kappa}}
        {n^{n+1}}
        =\infty,
    \end{align*}
    where the divergence is caused by $\kappa>1$. Combining this with Tonelli's theorem, we get
    \begin{gather*}
        \E\,\e^{\lambda|X_t|^\kappa}
        =\int_{\R^d}\E\,\e^{\lambda|X_t^{(1)}+y|^\kappa}\,
        \P\big(X^{(2)}_t\in\d y\big)
        =\infty.
    \end{gather*}

\medskip\noindent d)
    This follows from Theorem \ref{hb43}\,e) as we may choose $n\in\N$ such that $n\kappa\geq d$ and
    \begin{gather*}
        \E\,\e^{\lambda|X_t|^{-\kappa}}
        \geq \frac{\lambda^n}{n!}\,\E|X_t|^{-n\kappa}.
        \qedhere
    \end{gather*}
\end{proof}

\subsection{Subordinators}\label{sec32}

A subordinator is an increasing L\'{e}vy process in $\R$. Let $S=(S_t)_{t\geq0}$ be a subordinator with Bernstein function $\phi$ given by \eqref{bern}. Since $S$ is a L\'{e}vy process, all results of Subsection~3.1 hold with $X$ replaced by $S$.

The following example shows that the
result in Theorem \ref{hb43}\,b) is sharp for
Gamma subordinators.

\begin{example}\label{exam}
    Let $S=(S_t)_{t\geq0}$ be the Gamma process with parameters $\alpha,\beta>0$; this is a subordinator with
    $$
        b=0,
        \quad
        \nu(\d y) = \alpha y^{-1}\e^{-\beta y} \I_{(0,\infty)}(y)\,\d y.
    $$
    It is known that the distribution of $S_t$ at time $t>0$ is a $\Gamma(\alpha t,\beta)$-distribution, i.e.
    $$
        \P(S_t\in \d x)
        =\frac{\beta^{\alpha t}}{\Gamma(\alpha t)}\, x^{\alpha t-1}\e^{-\beta x}\I_{(0,\infty)}(x)\,\d x.
    $$
    Let $\kappa\in(0,1]$. Then we have
    $$
        G(t)
        :=\E S_t^\kappa=\frac{\Gamma(\alpha t+\kappa)}{\beta^\kappa\Gamma(\alpha t)}
    $$
    and
    $$
        H(t)
        :=\inf_{\theta\in[\kappa,1]}\left[t \int_{y>0}y^\theta\,\nu(\d y)\right]^{\kappa/\theta}
        = \frac{1}{\beta^\kappa} \inf_{\theta\in[\kappa,1]}\left[\alpha t\Gamma(\theta)\right]^{\kappa/\theta}.
    $$
    It is easy to check that
    $$
        \inf_{\theta\in[\kappa,1]}
        \left[\alpha t\Gamma(\theta)\right]^{\kappa/\theta}
        =
        \begin{cases}
            \alpha\Gamma(\kappa)t, &\text{if $t$ is small enough,}\\
            (\alpha t)^\kappa, &\text{if $t$ is large enough.}
        \end{cases}
    $$
    Since
    $$
        \lim_{t\downarrow 0}\frac{G(t)}{H(t)}
        = \lim_{t\downarrow 0}\frac{\Gamma(\alpha t+\kappa)}{\Gamma(\alpha t)\alpha\Gamma(\kappa)t}
        = \frac{1}{\Gamma(\kappa)} \lim_{t\downarrow0}\frac{\Gamma(\alpha t+\kappa)}{\Gamma(\alpha t+1)}
        =1,
    $$
    the upper bound in \eqref{ppb5c3} is sharp for small $t$. Moreover, by Stirling's formula
    \begin{equation}\label{stir}
        \Gamma(x)
        \sim\sqrt{2\pi}x^{x-\frac12}\e^{-x},\quad x\to\infty,
    \end{equation}
    one has
    $$
        \lim_{t\to\infty}\frac{G(t)}{H(t)}
        =\lim_{t\to\infty}\frac{\Gamma(\alpha t+\kappa)}{\Gamma(\alpha t)(\alpha t)^\kappa}
        =\frac{1}{\e^\kappa}\lim_{t\to\infty}\left(1+\frac{\kappa}{\alpha t}\right)^{\alpha t+\kappa-\frac12}
        =1.
    $$
    This means that \eqref{ppb5c3} is also sharp as $t\to\infty$.
\end{example}

We will need the following Blumenthal--Getoor index for subordinators
$$
    \sigma_\infty:=\inf\left\{\alpha\geq 0\,:\,\lim_{u\to\infty} \frac{\phi(u)}{u^\alpha}=0\right\}.
$$
Comparing this index with $\rho_\infty$ defined in \eqref{be1}, it is easy to see that $0\leq\rho_\infty\leq\sigma_\infty\leq1$. It is well known, cf.~\cite{BG61}, that
$$
    \sigma_\infty
    \geq\inf\left\{\alpha\geq 0\,:\,\int_{(0,1)}y^\alpha\,\nu(\d y)<\infty\right\}
$$
with equality holding for drift-free subordinators, i.e.\ $b=0$ in \eqref{bern}. It is also not hard to see, cf.~\eqref{ght56jn}, that
$$
    \sigma_\infty
    = \inf\left\{\alpha\geq 0\,:\,\limsup_{u\to\infty}\frac{\phi(u)}{u^\alpha}<\infty\right\}.
$$

Note that we can extend Bernstein functions
analytically onto the right complex half-plane
$\{z\in\C\,:\,\RE z>0\}$ and continuously onto
its closure $\{z\in\C\,:\,\RE z\geq0\}$,
see \cite[Proposition 3.6]{SSV}. Then the L\'{e}vy symbol of $S$ is
given by $\psi(\xi)=\phi(-\i\xi)$ for $\xi\in\R$. Since $\frac{\e-1}{\e}(1\wedge x)\leq
1-\e^{-x}\leq1\wedge x$ for all $x\geq0$ and
$\left|1-\e^{\i x}\right|\leq2\wedge|x|$ for all $x\in\R$, we
have for all $\xi\in\R\setminus\{0\}$
\begin{align*}
    \left|\phi(-\i\xi)\right|&\leq
    b|\xi|+\int_{(0,\infty)}\left|
    1-\e^{\i\xi x}\right|\,\nu(\d x)\\
    &\leq2\left(b|\xi|+\int_{(0,\infty)}
    \left(1\wedge\left[|\xi|x\right]\right)\,\nu(\d x)\right)\\
    &\leq\frac{2\e}{\e-1}\phi(|\xi|).
\end{align*}
This implies that for any $\alpha>0$
$$
    \limsup_{u\downarrow0}\frac{\phi(u)}{u^{\alpha}}<\infty
    \quad\Longrightarrow\quad
    \limsup_{|\xi|\to0}\frac{\left|\phi(-\i\xi)\right|}
    {|\xi|^{\alpha}}<\infty
$$
and
$$
    \limsup_{u\to\infty}\frac{\phi(u)}{u^{\alpha}}<\infty
    \quad\Longrightarrow\quad
    \limsup_{|\xi|\to\infty}\frac{\left|\phi(-\i\xi)\right|}
    {|\xi|^{\alpha}}<\infty.
$$
Thus, the following result is a direct consequence of
Theorem \ref{hb43}\,c).

\begin{corollary}\label{hgc}
    Assume that $\kappa\in(0,\sigma)$, where
    $\sigma\in(0,\sigma_0]$
    and $\limsup_{u\downarrow0}
    \phi(u)u^{-\sigma}<\infty$.\footnote{This is
    equivalent to either $0<\sigma<\sigma_0$ or $\sigma=\sigma_0>0$
    and $\limsup_{u\downarrow0}\phi(u)u^{-\sigma_0}<\infty$.}
    Then
    $$
        \E S_t^\kappa\leq C_{\kappa,\sigma}
        (t\vee1)^{\kappa/\sigma},
        \quad t>0.
    $$
    If furthermore $\limsup_{u\to\infty}
        \phi(u)u^{-\sigma}<\infty$,\footnote{This is
        equivalent to
    either $\sigma_\infty<\sigma\leq\sigma_0$ or
    $\sigma=\sigma_\infty\in(0,\sigma_0]$ and $\limsup_{u\to\infty}
    \phi(u)u^{-\sigma_\infty}<\infty$.}
    then
    $$
        \E S_t^\kappa\leq
        C_{\kappa,\sigma}\,t^{\kappa/\sigma},
        \quad t>0.
    $$
\end{corollary}

\begin{remark}
    As in Remark \ref{jn65cv}\,d), it is easy to see that Corollary \ref{hgc} is sharp for the $\alpha$-stable subordinator, $0<\alpha<1$.
\end{remark}

Recall that $\sigma_0$ is defined
by \eqref{be0}. Let
$$
    \rho_0:=\inf\left\{\alpha\,:\,\liminf_{u\downarrow0}
    \frac{\phi(u)}{u^\alpha}>0\right\}
    =\sup\left\{\alpha\,:\,\liminf_{u\downarrow0}
    \frac{\phi(u)}{u^\alpha}=0\right\}.
$$
Because of \eqref{hg321cf}, it is easy to see that $0\leq\sigma_0\leq\rho_0\leq1$.

The following result is essentially due to \cite{GRW}. For the sake of completeness, we present the argument. The proof is based on the fact that the functions $x\mapsto x^{-\kappa}$
and $x\mapsto\e^{\lambda x^{-\kappa}}$, $\kappa,\lambda,x>0$, are completely monotone functions, cf.~\cite[Chapter 1]{SSV}.

\begin{theorem}\label{hyv3}
    \textup{\bfseries a)}
            Let $\rho_\infty>0$ and $\kappa>0$.
            If for some $\rho\in(0,\rho_\infty]$, $\liminf_{u\to\infty}\phi(u)
            u^{-\rho}>0$,\footnote{This
            is equivalent to either $0<\rho<\rho_\infty$ or
            $\rho=\rho_\infty>0$ and
            $\liminf_{u\to\infty}\phi(u)u^{-\rho_\infty}>0$.}
            then for all $t>0$
            $$
                \E S_t^{-\kappa}
                \leq\frac{C_{\kappa,\rho}}{(t\wedge1)^{\kappa/\rho}}.
            $$
            If furthermore $\liminf_{u\downarrow0}\phi(u)
            u^{-\rho}>0$,\footnote{This
            is equivalent to either $\rho_0<\rho\leq\rho_\infty$
            or $\rho=\rho_0\in(0,\rho_\infty]$ and
            $\liminf_{u\downarrow0}\phi(u)u^{-\rho_0}>0$.} then for all $t>0$
            $$
                \E S_t^{-\kappa}
                \leq\frac{C_{\kappa,\rho}}{t^{\kappa/\rho}}.
            $$

    \medskip\noindent\textup{\bfseries b)}
            Let $\rho_\infty>0$ and $\kappa\in\left(0,\rho_\infty/(1-\rho_\infty)\right)$.
            If $\liminf_{u\to\infty} \phi(u)u^{-\rho}>0$,\footnote{This is equivalent
            to either $\kappa/(1+\kappa)<\rho<\rho_\infty$
            or $\rho=\rho_\infty>\kappa/(1+\kappa)$
            and $\liminf_{u\to\infty}
            \phi(u)u^{-\rho_\infty}>0$.} for some $\rho\in\left(\kappa/(1+\kappa),\rho_\infty\right]$, then for all $\lambda,t>0$
            $$
                \E\,\e^{\lambda S_t^{-\kappa}}
                \leq
                \exp\left[C_{\kappa,\rho}
                    \left(\lambda + \left(\frac{\lambda}{t^{\kappa/\rho}}
                    \right)^{\frac{\rho}{\rho -(1-\rho)\kappa}}
                    +\frac{\lambda}{t^{\kappa/\rho}}\right)
                \right].
            $$
            In particular, for any $\lambda>0$
            $$
                \E\,\e^{\lambda S_t^{-\kappa}}
                \leq
                \exp\left[\frac{C_{\kappa,\rho,\lambda}}
                {(t\wedge1)^{\frac{\kappa}{\rho-
                (1-\rho)\kappa}}}
                \right],\quad t>0.
            $$
            If furthermore $\liminf_{u\downarrow0}\phi(u)
            u^{-\rho}>0$,\footnote{This
            is equivalent to either
            $\frac{\kappa}{1+\kappa}\vee\rho_0<\rho\leq\rho_\infty$
            or $\rho=\rho_0\in\left(\kappa/(1+\kappa),
            \rho_\infty\right]$ and $\liminf_{u\downarrow0}
            \phi(u)u^{-\rho_0}>0$.} then for all $\lambda,t>0$
            $$
                \E\,\e^{\lambda S_t^{-\kappa}}
                \leq
                \exp\left[C_{\kappa,\rho}
                    \left( \left(\frac{\lambda}{t^{\kappa/\rho}}
                    \right)^{\frac{\rho}{\rho -(1-\rho)\kappa}}
                    +\frac{\lambda}{t^{\kappa/\rho}}\right)
                \right].
            $$

    \medskip\noindent\textup{\bfseries c)}
        Let $\sigma_\infty<1$. If $\kappa>\sigma_\infty/(1-\sigma_\infty)$,
        then for all $\lambda,t>0$
        $$
           \E\,\e^{\lambda S_t^{-\kappa}}
           =\infty.
        $$
\end{theorem}

\begin{remark}
       Let $S$ be an $\alpha$-stable subordinator ($0<\alpha<1$). Then
       we can choose in Theorem \ref{hyv3}\,a)
       $\rho=\rho_\infty=\rho_0=\alpha$.
       Note that $\E S_t^{-\kappa}=C_{\kappa,\alpha}\,t^{-\kappa/\alpha}$
       for all $\kappa,t>0$, see e.g.\ \cite[(25.5)]{Sato}.
       This means that Theorem \ref{hyv3}\,a) is sharp
       for $\alpha$-stable subordinators.
\end{remark}

\begin{proof}[Proof of Theorem \ref{hyv3}]
a)
    If $\liminf_{u\to\infty}\phi(u)
            u^{-\rho}>0$, then there exist constants $C_{1}=C_1(\rho)>0$ and $C_{2}=C_2(\rho)\geq0$ such that
    \begin{equation}\label{otr7}
        \phi(u)\geq C_{1}u^\rho,\quad u\geq C_{2}.
    \end{equation}
    Combining this with \eqref{h6fc}, we obtain
    \begin{equation}\label{fd332sd}
    \begin{aligned}
        \E S_t^{-\kappa}
        &= \frac{1}{\Gamma(\kappa)}\int_0^\infty u^{\kappa-1}\e^{-t\phi(u)}\,\d u\\
        &\leq \frac{1}{\Gamma(\kappa)}\int_0^{C_2}u^{\kappa-1}\,\d u + \frac{1}{\Gamma(\kappa)}\int_0^\infty u^{\kappa-1}\e^{-tC_1u^\rho}\,\d u\\
        &= \frac{C_2^\kappa}{\kappa\Gamma(\kappa)} + \frac{\Gamma\left(\frac\kappa\rho\right)}
        {\rho\Gamma(\kappa)(tC_1)^{\frac\kappa\rho}},
    \end{aligned}
    \end{equation}
    which implies the first assertion.
    If furthermore $\liminf_{u\downarrow0}\phi(u)
            u^{-\rho}>0$, then \eqref{otr7} holds with
    $C_2=0$. Thus, the second estimate follows
    from \eqref{fd332sd} with $C_2=0$.

\medskip\noindent b)
    It follows from \eqref{h6fc} that for $x\geq 0$
    \begin{align*}
        \e^{\lambda x^{-\kappa}}
        =1+\sum_{n=1}^\infty\frac{\lambda^n}{n!}\,\frac{1}{x^{n\kappa}}
        &=1+\sum_{n=1}^\infty\frac{\lambda^n}{n!}\frac{1}{\Gamma(n\kappa)}
        \int_0^\infty u^{n\kappa-1}\e^{-ux}\,\d u\\
        &=1+\int_0^\infty\e^{-ux}k(u)\,\d u,
    \end{align*}
    where
    $$
        k(u)
        :=\sum_{n=1}^\infty\frac{\lambda^n}{n!\Gamma(n\kappa)}\,u^{n\kappa-1},
        \quad u>0.
    $$
    Now we can use Tonelli's theorem to obtain
    \begin{equation}\label{ghy}
        \E\,\e^{\lambda S_t^{-\kappa}}
        =1+\E\left[\int_0^\infty\e^{-uS_t}k(u)\,\d u\right]
        =1+\int_0^\infty\e^{-t\phi(u)}k(u)\,\d u.
    \end{equation}
    If $\liminf_{u\to\infty}\phi(u)
            u^{-\rho}>0$, then by \eqref{otr7},
    \begin{align*}
        \E\,\e^{\lambda S_t^{-\kappa}}
        &\leq 1 + \int_0^{C_{2}} k(u)\,\d u + \int_{0}^\infty
        \e^{-C_{1}tu^\rho}
        k(u)\,\d u\\
        &= 1 + \sum_{n=1}^\infty \frac{\lambda^nC_2^{n\kappa}}{n!\,\Gamma(n\kappa)n\kappa} + \frac{1}{\rho}
        \sum_{n=1}^\infty \frac{\lambda^n\Gamma\left(\frac{n\kappa}{\rho}
        \right)}{n!\,\Gamma(n\kappa) \left(tC_1\right)^{\frac{n\kappa}{\rho}}}.
    \end{align*}
    Combining this with the inequalities
    \begin{gather*}
        \sqrt{2\pi}\,x^{x-\frac12}\e^{-x}
        \leq \Gamma(x)
        \leq \sqrt{2\pi}\,x^{x-\frac12}\e^{-x+\frac{1}{12x}},\quad x>0,\\
        n!\geq\sqrt{2\pi}\,n^{n+\frac12}\e^{-n},\quad n\in\N,
    \end{gather*}
    we arrive at
    \begin{align*}
        \E\,\e^{\lambda S_t^{-\kappa}}
        &\leq 1 + \frac{1}{2\pi\sqrt{\kappa}}\sum_{n=1}^\infty
        \frac{\left(\lambda C_2^\kappa\e^{\kappa+1}\kappa^{-\kappa}\right)^n}{nn^{(1+\kappa)n}}\\
        &\qquad\mbox{}+\frac{1}{\sqrt{2\pi\rho}}\sum_{n=1}^\infty
        \frac{\e^{\rho/(12n\kappa)}}{\sqrt{n}}\,
        n^{\left(\frac{\kappa}{\rho}-\kappa-1\right)n}
        \left(\frac{\e^{\kappa+1}}{\kappa^\kappa}\left(\frac{\kappa}{\rho\e C_1}\right)^{\kappa/\rho}
        \frac{\lambda}{t^{\kappa/\rho}}\right)^n\\
        &\leq 1 + \frac{1}{2\pi\sqrt{\kappa}}\sum_{n=1}^\infty
        \frac{\left(\lambda C_2^\kappa\e^{\kappa+1}\kappa^{-\kappa}\right)^n}{n^{(1+\kappa)n}}\\
        &\qquad\mbox{} + \frac{\e^{\rho/(12\kappa)}}{\sqrt{2\pi\rho}} \sum_{n=1}^\infty
        n^{\left(\frac{\kappa}{\rho}-\kappa-1\right)n}
        \left(\frac{\e^{\kappa+1}}{\kappa^\kappa}
        \left(\frac{\kappa}{\rho\e C_1}\right)^{\kappa/\rho}
        \frac{\lambda}{t^{\kappa/\rho}}\right)^n.
    \end{align*}
    Set $G:=\lambda C_2^\kappa\e^{\kappa+1}\kappa^{-\kappa}$; because of
    \begin{equation}\label{iu7y}
        n^n\geq n!,\quad n\in\N,
    \end{equation}
    we have
    $$
        \sum_{n=1}^\infty\frac{G^n}{n^{(1+\kappa)n}}
        \leq \sum_{n=1}^\infty\frac{G^n}{n^{n}}
        \leq \sum_{n=1}^\infty\frac{G^n}{n!}
        =\e^G-1.
    $$
    Set
    $$
        \epsilon
        :=-\frac{\kappa}{\rho}+\kappa+1\in(0,1]
        \quad\text{and}\quad
        H
        :=\frac{\e^{\kappa+1}}{\kappa^\kappa}\left(\frac{\kappa}{\rho\e C_1}\right)^{\kappa/\rho} \frac{\lambda}{t^{\kappa/\rho}}>0.
    $$
    Using Jensen's inequality and \eqref{iu7y}, it holds that
    \begin{align*}
        \sum_{n=1}^\infty\frac{H^n}{n^{\epsilon n}}
        =\sum_{n=1}^\infty\left(\frac{(2H)^{\frac{n}{\epsilon}}}{n^n}\right)^\epsilon\frac{1}{2^n}
        &\leq \left(\sum_{n=1}^\infty\frac{(2H)^{\frac{n}{\epsilon}}}{n^n}\frac{1}{2^n}\right)^\epsilon\\
        &\leq \left(\sum_{n=1}^\infty\frac{1}{n!}\left(\frac{(2H)^{\frac{1}{\epsilon}}}{2}\right)^n\right)^\epsilon\\
        &= \left(\exp\left[\frac{(2H)^{\frac{1}{\epsilon}}}{2}\right]-1\right)^\epsilon.
    \end{align*}
    Combining the above estimates and using the following elementary inequalities
    $$
        1+z(\e^x-1)^y\leq\e^{2xy+z(1+z)x^y}
        \quad\text{and}\quad
        \tfrac12\e^x+\tfrac12\e^y\leq\e^{x+y},\quad x,y,z\geq0,
   $$
    we obtain
    \begin{align*}
        \E&\,\e^{\lambda S_t^{-\kappa}}
        \leq 1
        + \frac{1}{2\pi\sqrt{\kappa}}\left(\e^G-1\right)
        +\frac{\e^{\rho/(12\kappa)}}{\sqrt{2\pi\rho}}
        \left(\exp\left[(2H)^{\frac{1}{\epsilon}}/2\right]-1\right)^\epsilon\\
        &= \frac12\left(1+\frac{1}{\pi\sqrt{\kappa}}\left(\e^G-1\right)\right)
        + \frac12\left(1+\frac{2\e^{\rho/(12\kappa)}}{\sqrt{2\pi\rho}}
        \left(\exp\left[(2H)^{\frac{1}{\epsilon}}/2\right]-1\right)^\epsilon\right)\\
        &\leq \frac12\exp\left[\left(2+\frac{1}{\pi^2\kappa}+\frac{1}{\pi\sqrt{\kappa}}\right)G\right]
        +\frac12\exp\left[(2H)^{\frac1\epsilon}\epsilon + \frac{2\e^{\rho/(12\kappa)}}{\sqrt{2\pi\rho}}
        \left(\frac{2\e^{\rho/(12\kappa)}}{\sqrt{2\pi\rho}}+1
        \right)\frac{2H}{2^\epsilon}\right]\\
        &\leq \exp\left[\left(2+\frac{1}{\pi^2\kappa}+\frac{1}{\pi\sqrt{\kappa}}\right)G
        + \epsilon(2H)^{\frac1\epsilon}
        + 2^{2-\epsilon}\frac{\e^{\rho/(12\kappa)}}{\sqrt{2\pi\rho}}
        \left(\frac{2\e^{\rho/(12\kappa)}}{\sqrt{2\pi\rho}}+1\right)
        H\right].
    \end{align*}
    Thus, the first assertion follows.
    If furthermore $\liminf_{u\downarrow0}\phi(u)
    u^{-\rho}>0$, then we can take $C_2=0$ and so $G=0$. Hence, the desired
    assertion follows from the above estimate with $G=0$.

\medskip\noindent c)
    Pick $\sigma\in\left(\sigma_\infty,\kappa/(1+\kappa)\right)$.
    By the definition of $\sigma_\infty$, there exist two positive constants $C_3=C_3(\sigma)$ and $C_4=C_4(\sigma)$ such that
    $$
        \phi(u)\leq C_3u^\sigma,\quad u\geq C_4.
    $$
    Together with \eqref{ghy} this yields that
    \begin{align*}
        \E\,\e^{\lambda S_t^{-\kappa}}
        &\geq \int_{C_4}^\infty \e^{-t\phi(u)}k(u)\,\d u\\
        &\geq \int_{C_4}^\infty \e^{-C_3tu^\sigma}k(u)\,\d u\\
        &= \sum_{n=1}^\infty\frac{\lambda^n}{n!\,\Gamma(n\kappa)
        \left(C_3t\right)^{\frac{n\kappa}{\sigma}}}
        \int_{C_3C_4^\sigma t}^\infty u^{\frac{n\kappa}{\sigma}-1}\e^{-u}\,\d u\\
        &= \sum_{n=1}^\infty\frac{\lambda^n}{n!\,\Gamma(n\kappa)
        \left(C_3t\right)^{\frac{n\kappa}{\sigma}}}
        \left[\Gamma\left(\frac{n\kappa}{\sigma}\right)
        - \int_0^{C_3C_4^\sigma t}u^{\frac{n\kappa}{\sigma}-1}\e^{-u}\,\d u\right]\\
        &\geq \sum_{n=1}^\infty\frac{\lambda^n}{n!\,\Gamma(n\kappa)
        \left(C_3t\right)^{\frac{n\kappa}{\sigma}}}
        \left[\Gamma\left(\frac{n\kappa}{\sigma}\right)
        - \frac{\sigma}{n\kappa} \left(C_3C_4^\sigma t\right)^{\frac{n\kappa}{\sigma}}\right].
    \end{align*}
    By \eqref{stir} and
    $$
        n!\sim\sqrt{2\pi}n^{n+\frac12}\e^{-n},
        \quad n\to\infty,
    $$
    we have for large $n\in\N$
    \begin{align*}
        &\frac{\lambda^n}{n!\,\Gamma(n\kappa)
        \left(C_3t\right)^{\frac{n\kappa}{\sigma}}}
        \left[\Gamma\left(\frac{n\kappa}{\sigma}\right)
        -\frac{\sigma}{n\kappa}\left(C_3C_4^\sigma t\right)^{\frac{n\kappa}{\sigma}}\right]\\
        &\quad \sim
        \frac{\lambda^n}{2\pi\kappa^{-\frac12}n^{(1+\kappa)n}
        \big(\e^{-\kappa-1}\kappa^\kappa\left(C_3t\right)
        ^{\frac{\kappa}{\sigma}}\big)^n}
        \left[\sqrt{\frac{2\pi\sigma}{n\kappa}}\, n^{\frac{n\kappa}{\sigma}}
        \left(\frac{\kappa}{\sigma\e}\right)^{\frac{n\kappa}{\sigma}}
        - \frac{\sigma}{n\kappa}\left(C_3C_4^\sigma t\right)^{\frac{n\kappa}{\sigma}}\right]\\
        &\quad\sim \frac{\lambda^n}{2\pi\kappa^{-\frac12}n^{(1+\kappa)n} \big(\e^{-\kappa-1}\kappa^\kappa
        \left(C_3t\right)^{\frac{\kappa}{\sigma}}\big)^n}
        \,\sqrt{\frac{2\pi\sigma}{n\kappa}} \,n^{\frac{n\kappa}{\sigma}}
        \left(\frac{\kappa}{\sigma\e}\right)^{\frac{n\kappa}{\sigma}}\\
        &\quad= \sqrt{\frac{\sigma}{2\pi}}\,\frac{1}{\sqrt{n}}\,
        n^{\left(-1-\kappa+\frac{\kappa}{\sigma}\right)n}
        \left(\frac{\lambda}{\e^{-\kappa-1}\kappa^\kappa}
        \left(\frac{\kappa}{\sigma\e C_3t}\right)^{\frac{\kappa}{\sigma}}\right)^n,
    \end{align*}
    which, because of $-1-\kappa+\frac{\kappa}{\sigma}>0$, tends to infinity as $n\to\infty$.
\end{proof}

\section{Appendix}

In this section, we show that
the index $\beta_0$ defined
by \eqref{bet} takes values in $[0,2]$.

Without loss of generality we may assume
that the L\'{e}vy measure $\nu\neq 0$; otherwise,
the assertion is trivial. Since
$$
    1-\cos u\geq(1-\cos1)u^2,\quad 0\leq u\leq1,
$$
we have for all $\xi\in\R^d$ that
\begin{align*}
    |\psi(\xi)|
    \geq \RE\psi(\xi)
    &=\frac12\xi\cdot Q\xi+\int_{y\neq0}
      \left(1-\cos|\xi\cdot y|\right)\,\nu(\d y)\\
    &\geq \int_{y\neq 0,\;|\xi\cdot y|\leq 1} \left(1-\cos|\xi\cdot y|\right)\nu(\d y)\\
    &\geq(1-\cos1) \int_{y\neq 0,\;|\xi\cdot y|\leq 1} |\xi\cdot y|^2\,\nu(\d y).
\end{align*}
Because $\nu\neq 0$, we know that there exists
a unit vector $x_0\in\R^d$ such that $\nu_D := \I_D\nu\neq0$, where
$$
    D:=\left\{z\in\R^d\setminus\{0\}:\,
    \arccos\frac{x_0\cdot z}{|z|}\in\left[0,\frac{\pi}{8}\right]\right\}.
$$
Since $\xi,y\in D$ satisfy
\begin{gather*}
    0
    \leq \arccos\frac{\xi\cdot y}{|\xi||y|}
    \leq \arccos\frac{\xi\cdot x_0}{|\xi|} + \arccos\frac{x_0\cdot y}{|y|}
    \leq \frac{\pi}{8}+\frac{\pi}{8}
    = \frac{\pi}{4},
\end{gather*}
we see
$$
    \frac{\xi\cdot y}{|\xi||y|}
    \geq \frac{1}{\sqrt{2}},\quad \xi,y\in D.
$$
Thus, we get for all $\xi\in D$ that
\begin{align*}
    \psi(\xi)
    &\geq(1-\cos1) \int_{y\in D,\;|\xi\cdot y|\leq 1} |\xi\cdot y|^2\,\nu(\d y)\\
    &\geq \frac{1-\cos1}{2}\,|\xi|^2 \int_{y\in D,\;|\xi\cdot y|\leq 1} |y|^2\,\nu(\d y),
\end{align*}
and, by Fatou's lemma,
\begin{align*}
    \liminf_{\xi\in D,\;|\xi|\to 0} \frac{|\psi(\xi)|}{|\xi|^2}
    &\geq\frac{1-\cos1}{2}\liminf_{\xi\in D,\;|\xi|\to0} \int_{\R^d}|y|^2\I_{\left\{|\xi\cdot y|\leq 1\right\}}\,\nu_D(\d y)\\
    &\geq\frac{1-\cos1}{2}\int_{\R^d}|y|^2 \liminf_{\xi\in D,\;|\xi|\to 0} \I_{\left\{|\xi\cdot y|\leq 1\right\}}\,\nu_D(\d y)\\
    &=\frac{1-\cos1}{2}\int_{\R^d}|y|^2\,\nu_D(\d y)\in(0,\infty].
\end{align*}
This shows that $\beta_0\leq2$.

\begin{ack}
    The authors would like to thank an anonymous referee for her/his insightful comments and useful suggestions which helped to improve the presentation of our manuscript.
\end{ack}

\end{document}